\documentclass[11pt,oneside,a4paper,reqno]{amsart}
\RequirePackage{fix-cm} 
\usepackage[T1]{fontenc}
\usepackage[english]{babel}
\usepackage{amssymb,amsmath,amsthm,graphicx,amsfonts}
\usepackage{lmodern}

\usepackage[dvipsnames]{xcolor}

\usepackage{enumerate}
\usepackage{hyperref}
\hypersetup{
colorlinks,
citecolor=Mahogany,
linkcolor=Mahogany,
urlcolor=Mahogany}
 
\usepackage{geometry}
\makeatletter
\def\@seccntformat#1{\hspace*{0mm}%
 \protect\textup{\protect\@secnumfont
   \ifnum\pdfstrcmp{subsection}{#1}=0 \bfseries\fi
   \csname the#1\endcsname
   \protect\@secnumpunct
     }%
}
\newcommand{\assign}{:=}
\newcommand{\mathD}{\mathrm{D}}
\newcommand{\mathLaplace}{\Delta}
\newcommand{\mathd}{\mathrm{d}}
\newcommand{\of}{:}
\newcommand{\suchthat}{:}
\newcommand{\tmSep}{; }
\newcommand{\tmabbr}[1]{#1}
\newcommand{\tmem}[1]{{\em #1\/}}
\newcommand{\tmmathbf}[1]{\ensuremath{\boldsymbol{#1}}}
\newcommand{\tmname}[1]{\textsc{#1}}
\newcommand{\tmop}[1]{\ensuremath{\operatorname{#1}}}
\newcommand{\tmsep}{, }
\newcommand{\tmstrong}[1]{\textbf{#1}}

\newcommand{\tmtextsf}[1]{\text{{\sffamily{#1}}}}

\newenvironment{enumerateroman}{\begin{enumerate}[i.] }{\end{enumerate}}
\newtheorem{lemma}{Lemma}
\newtheorem{proposition}{Proposition}
{\theoremstyle{remark}\newtheorem{remark}{Remark}}
\newtheorem{theorem}{Theorem}

\newcommand{\limi}[1]{\underset{#1}{\liminf}}
\newcommand{\limid}[1]{\liminf #1}
\newcommand{\lcup}{{\bigcup}}
\newcommand{\bp}{\,}
\newcommand{\Aop}{\mathsf{K}}
\newcommand{\Jop}{\mathsf{J}}
\newcommand{\Mop}{\mathsf{M}}
\newcommand{\Rop}{\mathsf{R}}
\newcommand{\vv}{\tmmathbf{v}}
\newcommand{\sss}{s}
\newcommand{\uu}{\tmmathbf{u}}
\newcommand{\calN}{\mathcal{N}}
\newcommand{\jmu}{\tmmathbf{\mu}}
\newcommand{\jM}{M}
\newcommand{\tmt}{\tmmathbf{\tau}}
\newcommand{\N}{N}
\newcommand{\Om}{\text{\tmtextsf{\ensuremath{\Omega}}}}
\newcommand{\T}{\top}

\newcommand{\m}{\tmmathbf{m}}

\newcommand{\grad}{\nabla}
\newcommand{\M}{\tmmathbf{M}}
\newcommand{\Stwo}{\mathbb{S}}
\newcommand{\RR}{\mathbb{R}}
\newcommand{\NN}{\mathbb{N}}
\newcommand{\eqs}{=}

\begin{document}

\title{Curved thin-film limits of chiral Dirichlet energies}

\author{G{\tmname{iovanni}} D{\tmname{i}} {\tmname{Fratta}}}
\address{Giovanni Di Fratta, Dipartimento di Matematica e Applicazioni ``R.
Caccioppoli'', Universit{\`a} degli Studi di Napoli ``Federico II'', Via
Cintia, Complesso Monte S. Angelo, 80126 Napoli, Italy}

\author{{\tmname{Valeriy Slastikov}}}
\address{Valeriy Slastikov, School of Mathematics, University of 
Bristol\\ Bristol BS8 1UG, United Kingdom}

\begin{abstract}
  We investigate the curved thin-film limit of a family of perturbed
  Dirichlet energies in the space of $H^1$ Sobolev maps defined in a tubular
  neighborhood of an $(n - 1)$-dimensional submanifold $\N$ of $\RR^n$ and
  with values in an $(m - 1)$-dimensional submanifold $\jM$ of $\RR^m$. The
  perturbation $\Aop$ that we consider is represented by a matrix-valued
  function defined on $\jM$ and with values in $\RR^{m \times n}$. Under
  natural regularity hypotheses on $\N$, $\jM$, and $\Aop$, we show that the
  family of these energies converges, in the sense of
  $\Gamma$-convergence, to an energy functional on $\N$ of an unexpected form,
  which is of particular interest in the theory of magnetic skyrmions. As a
  byproduct of our results, we get that in the curved thin-film limit,
  antisymmetric exchange interactions also manifest under an anisotropic term
  whose specific shape depends both on the curvature of the thin film and the
  curvature of the target manifold. Various types of
  antisymmetric exchange interactions in the variational theory of
  micromagnetism are a source of inspiration and motivation for our work.
\end{abstract}



{\maketitle}

\section{Introduction and motivation}

In this paper, we investigate the curved thin-film limit of a perturbed
Dirichlet energy of the form
\begin{equation}
  \mathcal{G}_{\varepsilon} \left( \vv \right) \assign \frac{1}{2 \varepsilon}
  \int_{\Om_{\varepsilon}} \left| D \vv (x) + \Aop \left( \vv (x) \right)
  \right|^2 \mathd x \label{eq:mainenfunc}
\end{equation}
defined on $H^1$ Sobolev maps $\vv : \Om_{\varepsilon} \rightarrow M$, where
for every $\varepsilon > 0$ the domain $\Om_{\varepsilon} \subseteq \RR^n$ is
the $\varepsilon$-tubular neighborhood of an $(n - 1)$-dimensional submanifold
$\N$ of $\RR^n$, and $M$ is an $(m - 1)$-dimensional submanifold of $\RR^m$.
The perturbation $\Aop$ is represented by a matrix-valued function defined on
$\jM$ and with values in $\RR^{m \times n}$ ($m$ rows, $n$ columns). We show
that under natural regularity hypotheses on $\N$, $\jM$, and $\Aop$,
the family $(\mathcal{G}_{\varepsilon})$ converges, in the
sense of $\Gamma$-convergence, to an energy functional on $\N$, which strongly
depends both on $N$ and $M$ and reveals remarkable physical implications to
relevant systems, for instance, to magnetic materials.

Our analysis shows that in the curved thin-film regime, when general base and
target manifolds are considered, generic antisymmetric exchange interactions
manifest themselves under an additional anisotropy term whose specific shape
depends both on the curvature of the thin film and the curvature of the target
manifold.

A specific form of \eqref{eq:mainenfunc}, namely the one emerging when $\N$ is
a {\tmem{planar}} surface in $\RR^2 \times \{ 0 \}$, $\jM \assign \Stwo^2$,
and $\Aop$ is the antisymmetric matrix $\Aop \left( \vv \right) = \bullet \,
\times \vv$, was recently addressed in the framework of
micromagnetics~{\cite{Davoli2022}}. This specific choice of $\N, \jM$, and
$\Aop$, describes a micromagnetic energy functional \eqref{eq:mainenfunc} that
accounts for antisymmetric interactions known as {\tmem{bulk
Dzyaloshinskii--Moriya interactions}} (DMI), whose presence explains the
emergence of chiral spin textures known as magnetic skyrmions
{\cite{Fert_2017,nagaosa2013topological}}. Even in this particular case, the
derived limiting model in {\cite{Davoli2022}} revealed intriguing physics: a
portion of the isotropic bulk DMI contributes to the shape anisotropy
originating from the magnetostatic self-energy. Our extension unveils
additional mechanisms even in the planar setting. Namely, what happens when
the target manifold is not really $\Stwo^2$, the DMI contribution is
anisotropic, or there are temperature variations in the ferromagnetic media
(see Subsection~\ref{subsec:physicalcontext}).

Given the importance of the energy $\mathcal{G}_{\varepsilon}$ in the analysis
of antisymmetric interactions, we will also refer to the energy
\eqref{eq:mainenfunc} as the {\tmem{chiral}} {\tmem{Dirichlet energy}}, even
when $\Aop$ is not antisymmetric.

\subsection{Outline}To adequately explain the ramifications of our findings,
we first review the physical framework that led us to the investigation of
\eqref{eq:mainenfunc}. In Subsection~\ref{subsec:stateofart}, we sum up
earlier research on the subject. In Section~\ref{subsec:rewriteG}, we describe
the rigorous setting of the problem and detail the contributions of the
present work. Proofs are given in Section~\ref{sec:proofmain}. In
Section~\ref{sec:applicationsMicromag}, we look at several applications of our
results to the variational theory of micromagnetism.

\subsection{Physical context}\label{subsec:physicalcontext}In single-crystal
ferromagnets, the observable magnetization states correspond to the local
minimizers of the micromagnetic energy functional, which, after normalization,
reads as~{\cite{BrownB1963,LandauA1935}}
\begin{equation}
  \mathcal{G}_{\tmop{sym}} (\m) \assign \mathcal{H}_{\tmop{ex}} \left( \m
  \right) +\mathcal{W} (\m) \assign \frac{1}{2}  \int_{\Om} | \nabla \m (x)
  |^2 \mathd x + \frac{1}{2}  \int_{\RR^3} \left| \grad u_{\m} (x) \right|^2
  \mathd x, \label{eq:vsrintercsmicromag}
\end{equation}
with $\m \in H^1 (\Om, \Stwo^2)$. The nonlocal term $\mathcal{W} (\m)$
represents the {\tmem{magnetostatic self-energy}} and describes the energy due
to the demagnetizing field $\grad u_{\m}$ generated by $\m \chi_{\Om}$, where
$\m \chi_{\Om}$ denotes the extension of $\m$ by zero in $\RR^3 \setminus
\Om$. The scalar potential $u_{\m}$ can be characterized as the unique
solution in $H^1 (\RR^3)$ of the Poisson equation
{\cite{BrownB1962,Di_Fratta_2019_var,praetorius2004analysis}}:
\begin{equation}
  - \mathLaplace u_{\m} = \tmop{div} \m \chi_{\Om} \quad \text{in }
  \mathcal{D}' \left( \RR^3 \right) . \label{eq:magnetostaticpotential}
\end{equation}
The {\tmem{exchange energy}} $\mathcal{H}_{\tmop{ex}}$ penalizes
nonuniformities in the orientation of the magnetization. Up to a constant
term, $\mathcal{H}_{\tmop{ex}}$ can be considered as the very-short-range
limit of a family of nonlocal Heisenberg models of the type
\begin{equation}
  \mathcal{H}_{\tmop{ex}}^{\varepsilon} \left( \m \right) \assign - \int_{\Om
  \times \Om} \omega_{\varepsilon} (| y - x |)  \m (x) \cdot \m (y) \mathd x
  \mathd y, \label{eq:Hexsymm}
\end{equation}
with $\omega_{\varepsilon}$ being a family of interaction kernels which tends to
concentrate around the origin ({\tmabbr{cf.}}~{\cite{Bourgain2001}}).

The interaction energy $\mathcal{H}_{\tmop{ex}}^{\varepsilon}$ is physically
{\tmem{symmetric}} in the {\tmem{lattice points}} $x, y$ where the magnetic
dipoles $\m (x)$, $\m (y)$ are located and, mathematically, this is captured
in the commutativity of the dot product in \eqref{eq:Hexsymm}. However,
magnetic materials with low crystallographic symmetry can exhibit a weak
{\tmem{antisymmetric}} exchange
interaction~{\cite{dzyaloshinsky1958thermodynamic,moriya1960anisotropic}}: a
relativistic effect caused by the spin-orbit coupling among neighbor magnetic
spins. This type of interaction was initially predicted and investigated by
Dzyaloshinskii~{\cite{dzyaloshinsky1958thermodynamic}} and
Moriya~{\cite{moriya1960anisotropic}}. For this reason, antisymmetric exchange
interactions are usually referred to as Dzyaloshinskii--Moriya interactions
(DMI).

The DMI induces a spin canting of the magnetic moments, whereas symmetric
exchange interactions favor parallel-aligned spins. Mimicking the expression
of the DMI energy associated with a discrete distribution of magnetic moments
({\tmabbr{cf. }}{\cite[p.~862]{Coey2021}}), one can assume that for a
continuous distribution of magnetic moments $\m \in L^2 \left( \Om, \Stwo^2
\right)$, after normalization, the overall antisymmetric interactions among
magnetic moments is described by the Hamiltonian function
\begin{equation}
  \mathcal{H}_{\tmop{DM}}^{\varepsilon} \left( \m \right) \assign \int_{\Om
  \times \Om} \tmmathbf{d}_{\varepsilon} (| y - x |) \cdot \left( \m (x)
  \times \m (y) \right) \mathd x \mathd y. \label{eq:nonlocDMI}
\end{equation}
The Dzyaloshinskii vector $\tmmathbf{d}_{\varepsilon}$ is an axial vector
that, other than from the relative distance between the magnetic moments $\m
(x)$ and $\m (y)$, also depends on the symmetry group of the crystal lattice;
its precise form has to be determined following Moriya's
rules~{\cite{moriya1960anisotropic}}.

Under appropriate assumptions on the $\varepsilon$-decay rate (and
concentration) of the family $(\tmmathbf{d}_{\varepsilon})_{\varepsilon > 0}$
in \eqref{eq:nonlocDMI}, one can obtain the very-short-range limiting model of
{\tmem{antisymmetric}} exchange interactions
\begin{equation}
  \mathcal{H}_{\tmop{DM}} \left( \m \right) = \sum_{i = 1}^3 \int_{\Om}
  \partial_i \m (x) \cdot \left( \tmmathbf{d}_i \times \m (x) \right) \mathd x
  \label{eq:H1asymDMI}
\end{equation}
with $(\tmmathbf{d}_i)_{i = 1}^3$ being suitable constant vectors that depend
on the symmetry of the crystal lattice. The energy $\mathcal{H}_{\tmop{DM}}$
in \eqref{eq:H1asymDMI} is, nowadays, the primary term used in the variational
theory of micromagnetics to explain the emergence of magnetic skyrmions.

Overall, in the presence of {\tmem{both}} symmetric and antisymmetric
interactions, the Heisenberg interaction energy, up to a constant factor, can
be expressed under the general form:
\begin{equation}
  \mathcal{H}_{\Om} \left( \m \right) \assign \frac{1}{2}  \sum_{i = 1}^3
  \int_{\Om} | \partial_{\tau_i (\xi)} \m (x) |^2 + 2\, \partial_i \m (x) \cdot
  \left( \tmmathbf{d}_i \times \m (x) \right) \mathd x.
  \label{eq:mainfunctionalinterest0}
\end{equation}
Therefore, if we denote by $\Jop$ the linear operator mapping the standard
basis $(\tmmathbf{e}_i)_{i = 1}^3$ into the vectors $(\tmmathbf{d}_i)_{i =
1}^3$, and denote by $\Mop \left( \m \right) \assign \bullet \, \times \m$ the
$\m$-dependent {\tmem{antisymmetric}} linear operator representing the action
of the cross-product with $\m$, then \eqref{eq:mainfunctionalinterest0} can be
rewritten as
\begin{align}
  \mathcal{H}_{\Om} \left( \m \right) & \eqs \frac{1}{2} \sum_{i = 1}^3
  \int_{\Om} \left| \partial_i \m (x) + \Mop \left( \m (x) \right) \, \Jop \,
  \tmmathbf{e}_i \right|^2 - \frac{1}{2} \sum_{i = 1}^3 \int_{\Om} \left| \,
  \Mop \left( \m (x) \right) \, \Jop \tmmathbf{e}_i \right|^2 \mathd x
  \nonumber\\
  & =  \frac{1}{2} \int_{\Om} \left| D \m (x) + \Mop \left( \m (x) \right)
  \, \Jop \right|^2 - \frac{1}{2} \int_{\Om} \left| \, \Mop \left( \m (x)
  \right)  \Jop \right|^2 \mathd x \nonumber\\
  & \eqs  \frac{1}{2} \int_{\Om} \left| D \m (x) + \Aop \left( \m (x)
  \right) \right|^2 - \frac{1}{2} \int_{\Om} \left| \, \Aop \left( \m (x)
  \right) \right|^2 \mathd x,  \label{eq:mainfunctionalinterest}
\end{align}
where we denoted by $\Aop \left( \m \right) \assign \Mop \left( \m \right)
\Jop$ the product of the $\m$-dependent antisymmetric linear operator $\Aop$
with the constant matrix $\Jop$.

It is worth noting that the second term in \eqref{eq:mainfunctionalinterest}
reflects a $\Gamma$-continuous term whose analysis in the thin-film regime is
straightforward. A similar remark applies to the magnetostatic self-energy
$\mathcal{W} $ investigated in {\cite{Di_Fratta_2019_var,Di_Fratta_2020}},
where it is shown that in the curved thin-film limit, the magnetostatic
self-energy $\mathcal{W}$ localizes to an easy-surface anisotropy in the
direction of the normal to the base manifold. Overall, it is sufficient to
study energy \eqref{eq:mainenfunc} to comprehend how DMI affects magnetization
behavior in curved thin layers.

\subsection{State of the art}\label{subsec:stateofart}The investigations of
Dirichlet-type energy functionals between manifolds is a subject with a long
history. Their minimization problem naturally arises in differential geometry
for studying unit-speed geodesics and minimal
surfaces~{\cite{Eells1995,Schoen1997}}. However, it also appears in nonlinear
field theories where certain forms of interactions are assimilated through the
ideas of elastic deformations. Outstanding examples are the elastic free
energy in the Oseen-Frank theory of nematic liquid crystals
{\cite{Ball_2017,Virga_2018}}, the exchange energy in the variational theory
of micromagnetism~{\cite{BrownB1963,Hubert1998}}, and string theory and
$M$-theory~{\cite{Albeverio_1997,Becker_2006}}. The first-order minimality
conditions of Dirichlet energy produce harmonic map equations, whose
systematic treatment gained impetus after the seminal work of Eells and
Sampson {\cite{Eells1964}}, who showed that in specific geometric contexts,
arbitrary maps could be deformed into harmonic maps. The overall subject is so
vast that it would be impossible here even to scratch the surface of the
literature on the topic; therefore, we refer to
{\cite{Eells1995,Schoen1997,Lin2008}} and references therein for further
details.

Here, instead, we give an overview of the literature intimately related to our
investigations in the curved thin film regime. In the variational theory of
micromagnetics, which consists of a nonlocal addendum to the theory of
$\Stwo^2$-valued harmonic maps, the story dates back to the seminal paper
{\cite{GioiaJames97}}, where the authors showed that in thin films, the
nonlocal effects of the so-called demagnetizing field operator localize to an
easy-surface anisotropy term. Later, various static and dynamic reduced
theories for thin-film micromagnetics has been established under different
scaling regimes
in~{\cite{Babadjian_2022,DeSimoneKohnMuellerOtto02,KnuepferMuratovNolte19,KohnSlastikov05,KohnSlastikov-dyn05,Melcher10,Melcher14, Moser04,Moser05}}.

Ferromagnetic systems in the shape of {\tmem{curved}} thin films are
currently of interest due to their capability to host {\tmem{spontaneous}}
skyrmion solutions, i.e.,
chiral spin textures that carry a non-trivial topological degree, even in
materials where the antisymmetric spin-orbit coupling mechanism (in the guise
of DMI) can be neglected. The evidence of these states sheds light on the role
of the geometry in magnetism: chiral spin-textures can be stabilized by
curvature effects only, in contrast to the planar case where DMI is required.
In addition to fundamental reasons, the interest in these geometries is
triggered by recent advances in the fabrication of magnetic spherical hollow
nanoparticles, which lead to artificial materials with unexpected
characteristics and numerous applications ranging from logic devices to
biomedicine~(see, e.g., the topical review {\cite{Streubel_2016}} and \cite{MaKSheka2022}). From
the mathematical point of view, the dimension reduction problems in non-planar
thin magnetic layers have been studied
in~{\cite{carbou2001thin,Di_Fratta_2019, Di_Fratta_2019_var, Di_Fratta_2020, Melcher_2019, MoriniSlastikov18,Slastikov05}}.

Dimension reduction results taking into account also antisymmetric exchange
interactions have been reported in the regime of planar thin film and for bulk
DMI, i.e., when in \eqref{eq:mainfunctionalinterest0} one considers
$\Om_{\varepsilon} \subseteq \RR^3$ of the type $\Om_{\varepsilon} \assign
\omega \times (- \varepsilon, \varepsilon)$ with $\omega$ an open subset of
$\RR^2$, and the classical $\Stwo^2$-valued magnetization $\m \in H^1(
\Om_{\varepsilon}, \Stwo^2)$, see {\cite[Theorem~1]{Davoli2022}}. The
surprising consequence of this result is that a portion of the bulk DMI energy
contributes an additional shape anisotropy term \ $\kappa^2  \int_{\omega} (\m
(\sigma) \cdot \tmmathbf{e}_3)^2$ \ in the limiting energy, to enhance the
shape anisotropy of the thin film, see
Subsection~\ref{subsec:isotropicbulkDMI} for extended discussion.

In this paper, we remove the rigid conditions of $\m$ being $\Stwo^2$-valued,
the thin film being planar, and treat a general perturbation $\Aop$. It allows
us to treat an interesting case of anisotropic
DMI~{\cite{Camosi_2017,Ga_2022}} and other engaging physical scenarios, such
as constrained $Q$-tensor theories of nematic liquid crystals as well their
Oseen-Frank counterpart in the presence of anisotropic bulk
potentials~{\cite{Ball_2017}}. It can also be applied to various limiting
Ginzburg--Landau-type models~{\cite{Bethuel_2017}}. \

An important contribution of our result to the physics of magnetic materials
comes from the fact that it uncovers yet another effective anisotropy term
(apart from one reported in {\cite{Davoli2022}}) coming from the interplay
between a ferromagnetic layer, target manifold for magnetization vector, and a
specific form of DMI. It reveals, for example, what happens when one considers
a ferromagnet in which the temperature is not necessarily uniform or DMI is
anisotropic. We refer to Section~\ref{sec:applicationsMicromag} for various
interesting examples in micromagnetics.

\section{Contributions of the present work}\label{subsec:rewriteG}

\subsection{Notation and setup}The main result of this paper concerns the
variational characterization of the asymptotic behavior of a rescaled version
of \eqref{eq:mainenfunc} in the limit for $\varepsilon \rightarrow 0$. To
state our results in a precise way, we need to set up the framework and
introduce notation.

We denote by $\N$ and $\jM$ closed (i.e., compact, without boundary, and
connected) hypersurfaces of class $C^2$, respectively, of $\RR^n$ and $\RR^m$.
It is well-known that as a consequence of the Jordan-Brouwer separation
theorem, $\N$ and $\jM$ are orientable~{\cite{Lima1988}}. The normal fields
associated with the choice of orientations of $\N$ and $\jM$ will be denoted,
respectively, by $\tmmathbf{n}_{\N} : \N \rightarrow \Stwo^{n - 1}$, and
$\tmmathbf{n}_{\jM} : \jM \rightarrow \Stwo^{m - 1}$.

To make the reading more comfortable, we consistently denote by $\xi$ a
generic point on $\N$ and by $\sigma$ a generic point on the target space
$\jM$. Also, we denote by $T_{\xi} \N$ and $T_{\sigma} \jM$, respectively, the
tangent space to $\N$ at $\xi$ and the tangent space to $\jM$ at $\sigma$, and
we use the notation $T \N$ and $T \jM$ to denote the corresponding tangent
bundles.

Our hypotheses on $\N$ and $\jM$ assure that they both admit a tubular
neighborhood (of uniform thickness). We recall the definition as it also
allows us to fix notation. Given a closed $C^2$-hypersurface $S \subseteq
\RR^n$ (for which Jordan-Brouwer separation theorem holds), we can denote by
$d_S : \RR^n \rightarrow \RR$ the {\tmem{signed distance}} from $S$, defined
by
\[ d_S (x) = \left\{ \begin{array}{ll}
     d (x, S) & \text{if } x \in S_+,\\
     - d (x, S) & \text{if } x \in S_-,
   \end{array} \right. \]
where we denoted by $d (x, S)$ the {\tmem{Euclidean}} {\tmem{distance}} of $x$
from $S$, by $S_+$ the outer (unbounded) component of $\RR^n \setminus S$, and
by $S_-$ the interior one. We say that the open set
\[ O_{\delta} \assign \left\{ x \in \RR^n : d_S (x) < \delta \right\} \]
is a {\tmem{tubular neighborhood}} of $S$ of uniform thickness $\delta > 0$ if
the following property holds ({\tmabbr{cf.}}~{\cite{Carmo2016}}): with
$\mathcal{S} \assign S \times I$, $I \assign (- 1, 1)$, one has that for every
$0 < \varepsilon < \delta$, the map
\begin{equation}
  \psi_{\varepsilon} : (\xi, \sss) \in \mathcal{S} \mapsto \xi +
  \varepsilon \sss \tmmathbf{n}_S (\xi) \in O_{\varepsilon}
  \label{eq:diffeomorphism}
\end{equation}
is a $C^1$-diffeomorphism of $\mathcal{S}$ onto $O_{\varepsilon}$. In this
case, the nearest point projection map
\begin{equation}
  \pi_S : O_{\varepsilon} \rightarrow S, \label{eq:npprojection}
\end{equation}
which maps any $x \in O_{\varepsilon}$ onto the unique $\pi_S (x) \in S$ such
that $x = \pi_S (x) + d_S (x) \tmmathbf{n}_S (\pi_S (x))$, is a map of class
$C^1$ and, therefore, so is $d_S (x) = (x - \pi_S (x)) \cdot \tmmathbf{n}_S
(\pi_S (x))$. Moreover, one has
\[ \nabla d_S (\pi_S (x)) =\tmmathbf{n}_S (\pi_S (x)) . \]
In what follows, we denote by $\Om_{\delta} \assign \left\{ x \in \RR^n
\suchthat d_{\N} (x) < \delta \right\}$ a tubular neighborhood of $\N$ of
thickness $\delta$ and by $O_{\delta}$ a tubular neighborhood of $\jM$. Also,
we set $\calN \assign \N \times I$, so that, from \eqref{eq:diffeomorphism},
$\Om_{\delta} \equiv \psi_{\delta} \left( \calN \right)$. Moreover, to shorten
the notation, for any $\delta > 0$ we set $I_{\delta} \assign (- \delta,
\delta)$ and $I_{\delta_{+}} \assign (0, \delta)$.

For every $\xi \in \N$ the symbols $\tmt_1 (\xi), \tmt_2 (\xi), \ldots,
\tmt_{n - 1} (\xi)$ are used to denote an orthonormal basis of $T_{\xi} \N$
made by its principal directions, i.e., an orthonormal basis consisting of
eigenvectors of the shape operator of $\N$ ({\tmabbr{cf.}},
e.g.,~{\cite{Carmo2016}}). We then write $\kappa_1 (\xi), \kappa_2 (\xi),
\ldots, \kappa_{n - 1} (\xi)$ for the principal curvatures at $\xi \in \N$.
Note that, for any $x \in \Om_{\delta}$ the frame
\begin{equation}
  \left( \tmt_1 (\xi), \tmt_2 (\xi), \ldots, \tmt_{n - 1} (\xi),
  \tmmathbf{n}_{\N} (\xi) \right) \quad \text{with} \quad \xi \assign \pi_{\N}
  (x) \;, \label{eq:convtrihedron}
\end{equation}
constitutes an orthonormal basis of $T_{\pi_{\N} (x)} \Om_{\delta}$ that
depends only on $\N$. Note that we do not specify the symbol $\N$ in the
notation related to the orthonormal basis of $T_{\xi} \N$. This creates no
confusion because we will not use orthonormal bases of $T_{\sigma} \jM$.

For $0 < \varepsilon < \delta$, we denote by
$\sqrt{\mathfrak{g}_{\varepsilon} (\xi, \sss)}$ the metric factor
which relates the volume form on $\Om_{\varepsilon}$ at $\xi + \varepsilon
\sss \tmmathbf{n}_{\N} (\xi)$ to the volume form on $\calN$ at $\left( \xi,
\sss \right)$. Also, we denote by $(\mathfrak{h}_{1, \varepsilon} \left( \xi,
\sss \right))_{i=1}^{n-1}$ the metric
coefficients which connect the tangential part of the gradient at $\xi +
\varepsilon \sss \tmmathbf{n}_{\N} (\xi) \in \Om_{\varepsilon}$ to the
tangential gradient on $\N$ at $(\xi, \sss)$.

A direct computation (see Lemma~\ref{eq:lemmametricfactor} in
Section~\ref{sec:proofmain}) shows that
\begin{equation}
  \sqrt{\mathfrak{g}_{\varepsilon} (\xi, \sss)} \assign \Pi_{i =
  1}^{n - 1} (1 + \varepsilon s \kappa_i (\xi)), \quad \mathfrak{h}_{i,
  \varepsilon} (\xi, s) \assign \frac{1}{1 + \varepsilon \sss \kappa_i (\xi)}
  \quad (i = 1, 2, \ldots, n - 1) .
\end{equation}
In what follows, without loss of generality, we shall always assume that the
thickness $\delta$ is sufficiently small so that for every $\varepsilon \in
I_{\delta+}$ there holds
\begin{equation}
  c_{\calN}^{- 1} \leqslant \sqrt{\mathfrak{g}_{\varepsilon} \left( \xi, \sss
  \right)} \leqslant c_{\calN}, \quad c_{\calN}^{- 1} \leqslant
  \mathfrak{h}_{i, \varepsilon} (\xi, \sss) \leqslant c_{\calN}
  \label{eqs:lowbounds},
\end{equation}
for some positive constant $c_{\calN} > 0$.

Also, we shall denote by $H^1( \calN, \RR^m)$ the Sobolev space
of vector-valued functions defined on $\calN$ (see, e.g., {\cite{Wloka1995}})
endowed with the norm
\begin{equation}
  \left\| \uu \right\|^2_{H^1( \calN, \jM)} \assign \int_{\calN}
  \left| \uu (\xi, \sss) \right|^2 \mathd \xi \mathd \sss +
  \int_{\calN} \left| \grad_{\xi}  \uu (\xi, \sss) \right|^2 +
  \left| \partial_{\sss} \uu (\xi, \sss) \right|^2 \mathd \xi
  \mathd \sss .
\end{equation}
Here, $\grad_{\xi}$ is the tangential gradient of $\uu$ on $\N$, and $\left|
\grad_{\xi}  \uu (\xi, \sss) \right|^2 = \sum_{i = 1}^{n - 1}
\left| \partial_{\tmmathbf{\tau}_i (\xi)} \uu (\xi, s) \right|^2$.

We write $H^1( \calN, \jM)$ for the subset of $H^1( \calN,
\RR^m)$ made by vector-valued functions with values in $\jM$, and we
use the notation $H^1( \calN, T \jM)$ when we want to emphasize
that the target manifold is the tangent bundle of $\jM$.

\subsection{The chiral Dirichlet energy}Let $\N$ and $\jM$ be closed and
smooth ($C^2$) hypersurfaces of $\RR^n$ and $\RR^m$, respectively. For every
$0 < \varepsilon < \delta$ we consider the ($n$-dimensional) tubular
neighborhood $\Om_{\varepsilon} = \left\{ x \in \RR^n \suchthat d_{\N} (x) <
\varepsilon \right\}$, and consider the family of energy functionals defined
for every $\vv \in H^1( \Om_{\varepsilon}, \jM)$ by
\begin{equation}
  \mathcal{G}_{\varepsilon} \left( \vv \right) = \frac{1}{2 \varepsilon}
  \int_{\Om_{\varepsilon}} \left| D \vv (x) + \Aop \left( \vv (x) \right)
  \right|^2 \mathd x, \label{eq:enfun}
\end{equation}
with $| \cdot |$ being the Euclidean norm on $\RR^{m \times n}$ ($m$ rows, $n$
columns) and where
\[ \Aop : \sigma \in \jM \mapsto \Aop (\sigma) \in \RR^{m \times n} \]
is a Lipschitz continuous function, i.e., there exists $c_{\Aop} > 0$ such
that
\begin{equation}
  \left| \Aop (\sigma_1) - \Aop (\sigma_2) \right|_{n \times m} \leqslant
  c_{\Aop} | \sigma_1 - \sigma_2 |_m \quad \forall \sigma_1, \sigma_2 \in \jM
  \label{eq:cALip} .
\end{equation}
\begin{remark}
  Since we are assuming $\jM$ to be bounded, \eqref{eq:cALip} implies that for
  any $\sigma_0 \in \jM$ there holds $\left| \Aop (\sigma) \right| \leqslant
  \left| \Aop (\sigma_0) \right| + c_{\Aop} \tmop{diam}( \jM)$,
  i.e., that the image of $\Aop$ is bounded. In what follows, to simplify the
  constants, we will assume that a Lipschitz constant $c_{\Aop}$ is chosen big
  enough so that there holds
  \begin{equation}
    \left| \Aop (\sigma) \right| \leqslant c_{\Aop} \quad \forall \sigma \in
    \jM \label{eq:cALipnew} .
  \end{equation}
\end{remark}

For any $\varepsilon \in I_{\delta+}$, the existence of at least a minimizer
for $\mathcal{G}_{\varepsilon}$ in $H^1( \Om_{\varepsilon}, \jM)$
is a simple application of the direct method of the calculus of variations. We
are interested in the asymptotic behavior of the family of minimizers of
$(\mathcal{G}_{\varepsilon})_{\varepsilon \in I_{\delta+}}$ as $\varepsilon
\rightarrow 0$.{\medskip}

Let us introduce the following functional defined on $H^1( \calN, \jM)$, which can be thought of as the pull-back of
$\mathcal{G}_{\varepsilon}$ on the product manifold $\calN \assign \N \times
I$, $I \assign (- 1, 1)$:
\begin{align}
  \mathcal{E}_{\calN}^{\varepsilon} \left( \uu_{\varepsilon} \right) & \assign
   \frac{1}{2} \int_{\calN} \sum_{i = 1}^{n - 1} \left| \mathfrak{h}_{i,
  \varepsilon} (\xi, \sss) \partial_{\tmt_i (\xi)}
  \uu_{\varepsilon} (\xi, \sss) + \Aop \left( \uu_{\varepsilon}
  (\xi, \sss) \right) \bp \tmt_i (\xi) \right|^2 
  \sqrt{\mathfrak{g}_{\varepsilon} (\xi, \sss)} \mathd \xi \mathd
  \sss \nonumber\\
  &   \qquad \qquad + \frac{1}{2} {\int_{\calN}}_{} \left| \frac{1}{\varepsilon}
  \partial_{\sss} \uu_{\varepsilon} (\xi, \sss) + \Aop \left(
  \uu_{\varepsilon} (\xi, \sss) \right) \bp \tmmathbf{n}_{\N}
  (\xi) \right|^2 \sqrt{\mathfrak{g}_{\varepsilon} (\xi, \sss)}
  \mathd \xi \mathd \sss .  \label{eq:rewriteG1}
\end{align}
Our main result is stated in the following statement.

\begin{theorem}
  \label{thm:main}For any $\varepsilon \in I_{\delta+}$, the minimization
  problem for $\mathcal{G}_{\varepsilon}$ in $H^1( \Om_{\varepsilon},
  \jM)$ is equivalent to the minimization in $H^1( \calN, \jM)$ of the functional $\mathcal{E}_{\calN}^{\varepsilon}$ defined by
  {\tmem{\eqref{eq:rewriteG1}}} in the sense that a configuration
  $\vv_{\varepsilon} \in H^1( \Om_{\varepsilon}, \jM)$ minimizes
  $\mathcal{G}_{\varepsilon}$ if and only if $\uu_{\varepsilon} \assign
  \vv_{\varepsilon} \circ \psi_{\varepsilon} \in H^1( \calN, \jM)$ minimizes $\mathcal{E}_{\calN}^{\varepsilon}$.
  
  The family $( \mathcal{E}_{\calN}^{\varepsilon} )_{\varepsilon
  \in I_{\delta+}}$ is equicoercive in the weak topology of $H^1(
  \calN, \jM)$ and the $\Gamma$-limit $\mathcal{E}_{\calN} \assign
  \Gamma \text{-} \lim_{\varepsilon \rightarrow 0}
  \mathcal{E}_{\calN}^{\varepsilon}$ is defined for every $\uu \in H^1(
  \calN, \jM)$ by
  \begin{equation}
    \mathcal{E}_{\calN} (\uu) = \left\{ \begin{array}{ll}\displaystyle
      \sum_{i = 1}^{n - 1} \int_{\N} \left| \partial_{\tmmathbf{\tau}_i (\xi)}
      \uu (\xi) + \Aop (\uu(\xi)) \bp \tmt_i (\xi) \right|^2  &
      \\
     \displaystyle \qquad \qquad + \int_{\N} \left( \Aop \left( \uu (\xi) \right)
      \tmmathbf{n}_{\N} (\xi) \cdot \tmmathbf{n}_{\jM} \left( \uu (\xi)
      \right) \right)^2 \mathd \xi & \text{if } \partial_{\sss} \uu = 0,\\
     \displaystyle + \infty & \text{otherwise} .
    \end{array} \right. \label{eq:gleF}
  \end{equation}
  Moreover,
  \begin{equation}
    \min_{H^1( \Om_{\varepsilon}, \jM)} \mathcal{G}_{\varepsilon}
   \, =\, \min_{H^1( \calN, \jM)} \mathcal{E}_{\calN}^{\varepsilon}
    \, = \, \min_{H^1( \calN, \jM)}  \mathcal{E}_{\calN} + o (1)\, ,
    \label{eq:firstordergammadevelopforF}
  \end{equation}
  and if $( \uu_{\varepsilon})_{\varepsilon \in I_{\delta+}}$ is
  a minimizing family for $( \mathcal{E}_{\calN}^{\varepsilon}
  )_{\varepsilon \in I_{\delta+}}$ then $( \uu_{\varepsilon}
  )_{\varepsilon \in I_{\delta+}}$ converges, strongly in $H^1(
  \calN, \jM)$, to a minimizer of $\mathcal{E}_{\calN}$.
\end{theorem}

\begin{remark}
  \label{rmk:curvinterf}The arguments we present to prove
  Theorem~\ref{thm:main} extend with minor modifications to the analysis of
  chiral Dirichlet energies of the form
  \begin{equation}
    \tilde{\mathcal{G}}_{\varepsilon} \left( \vv \right) \assign \frac{1}{2
    \varepsilon} \int_{\Om_{\varepsilon}} \left| A (x) D \vv (x) + \Aop \left(
    x, \vv (x) \right) \right|^2 \mathd x, \label{eq:mainenfuncgen}
  \end{equation}
  in which the operator $\Aop$ in \eqref{eq:enfun} depends on $x \in
  \Om_{\varepsilon}$ as well as $\sigma \in \jM$, \eqref{eq:cALip} and
  \eqref{eq:cALipnew} are assumed to hold uniformly in $x \in
  \Om_{\varepsilon}$, and with the tensor $A \in L^{\infty} \left(
  \Om_{\varepsilon}, \left. \RR^{m \times m} \right) \right)$ uniformly
  elliptic, i.e., such that for every $x \in \Omega_\varepsilon$ and every $y \in \RR^m
  \setminus \{ 0 \}$, there holds
  \begin{equation}
    \Lambda_A  | y |^2 \geqslant A (x) y \cdot y \geqslant \lambda_A \cdot | y
    |^2,
  \end{equation}
  for positive constants $\Lambda_A, \lambda_A > 0$ that do not depend on $x
  \in \Om_{\varepsilon}$. For that, one also assumes that both $\Aop \left(
  \cdot, \vv \right)$ and $A (\cdot)$ are uniformly continuous in the normal
  direction, i.e., in terms of a modulus of continuity $\varpi_{\Aop}$ and
  $\varpi_A$, such that $\left| \Aop (x, \sigma) - \Aop \left( \pi_{\N} (x),
  \sigma \right) \right| \leqslant \varpi_\Aop \left( \left| x - \pi_{\N} (x)
  \right| \right)$ and $\left| A (x) - A \left( \pi_{\N} (x) \right) \right|
  \leqslant \varpi_A \left( \left| x - \pi_{\N} (x) \right| \right)$ for every
  $x \in \Om_{\varepsilon}$. To make the reading more comfortable, we give the
  proof by focusing on the energy $\mathcal{G}_{\varepsilon}$ in
  Theorem~\ref{thm:main}. Nevertheless, the generalized energy
  \eqref{eq:mainenfuncgen} can be important in applications; therefore, we
  give details about the curved thin-film limit in this case. In stating the
  result, we will denote by $\tilde{\mathcal{E}}_{\calN}^{\varepsilon}$ the
  pull-back of $\tilde{\mathcal{G}}_{\varepsilon}$ from $H^1(
  \Om_{\varepsilon}, \jM)$ to $H^1( \calN, \jM)$ that one
  can obtain following the same steps that led to the pull-back
  $\eqref{eq:rewriteG1}$.
  \end{remark}
  \begin{theorem}
    \label{thm:maingen}For any $\varepsilon \in I_{\delta+}$, the
    minimization problem for $\tilde{\mathcal{G}}_{\varepsilon}$ in $H^1( \Om_{\varepsilon}, \jM)$ is equivalent to the minimization
    in $H^1( \calN, \jM)$ of its pull-back
    $\tilde{\mathcal{E}}_{\calN}^{\varepsilon}$, in the sense that a
    configuration $\vv_{\varepsilon} \in H^1( \Om_{\varepsilon}, \jM)$ minimizes $\tilde{\mathcal{G}}_{\varepsilon}$ if and only if
    $\uu_{\varepsilon} \assign \vv_{\varepsilon} \circ \psi_{\varepsilon} \in
    H^1( \calN, \jM)$ minimizes
    $\tilde{\mathcal{E}}_{\calN}^{\varepsilon}$.
  
    The family $( \tilde{\mathcal{E}}_{\calN}^{\varepsilon})_{\varepsilon \in I_{\delta+}}$ is equicoercive in the weak
    topology of $H^1( \calN, \jM)$ and the $\Gamma$-limit
    $\tilde{\mathcal{E}}_{\calN} \assign \Gamma \text{-} \lim_{\varepsilon
    \rightarrow 0} \tilde{\mathcal{E}}_{\calN}^{\varepsilon}$ is defined for
    every $\uu \in H^1( \calN, \jM)$ by
    \begin{equation}
      \tilde{\mathcal{E}}_{\calN} ( \uu) = \left\{
      \begin{array}{ll}\displaystyle
        \sum_{i = 1}^{n - 1} \int_{\N} \left| A (\xi)
        \partial_{\tmmathbf{\tau}_i (\xi)} \uu (\xi) + \Aop \left( \xi, \uu
        (\xi) \right) \bp \tmt_i (\xi) \right|^2 \mathd \xi & \\
        \displaystyle \qquad \qquad + \int_{\N} \left( \frac{A^{- 1} (\xi) \Aop
        \left( \xi, \uu (\xi) \right) \tmmathbf{n}_{\N} (\xi) \cdot
        \tmmathbf{n}_{\jM} \left( \uu (\xi) \right)}{A^{- 1} (\xi)
        \tmmathbf{n}_{\jM} (\uu(\xi)) \cdot \tmmathbf{n}_{\jM}
        \left( \uu (\xi) \right)}\right)^2 \mathd \xi \quad & \text{if }
        \partial_{\sss} \uu = 0,\\
        + \infty & \text{otherwise} .
      \end{array} \right. \label{eq:gleFgen}
    \end{equation}
    Moreover, if $( \uu_{\varepsilon})_{\varepsilon \in I_{\delta
    +}}$ is a minimizing family for $(
    \tilde{\mathcal{E}}_{\calN}^{\varepsilon} )_{\varepsilon \in
    I_{\delta+}}$ then $( \uu_{\varepsilon})_{\varepsilon \in
    I_{\delta+}}$ converges, strongly in $H^1( \calN, \jM)$, to
    a minimizer of $\tilde{\mathcal{E}}_{\calN}$.
  \end{theorem}

\begin{remark}
  \label{rmk:moregensurf}To make the analysis more pleasant to read, we
  assumed $\N$ and $\jM$ to be closed hypersurface of class $C^2$. However, as
  it will be transparent from the proofs, all the results hold as soon as $\N$
  and $\jM$ are bounded hypersurfaces (without boundary) of class $C^2$ that
  admit a tubular neighborhood (of uniform thickness). Therefore, the range of
  hypersurfaces $\N$ and $\jM$ included in our analysis is broad. Indeed,
  since any closed hypersurface of class $C^2$ admits a tubular neighborhood
  (of uniform thickness) {\cite[Prop.~1, p.~113]{Carmo2016}}, our analysis
  holds for classical convex surfaces like spheres, ellipsoids, planar
  surfaces, as well as for non-convex ones like the torus. The analysis also
  extends to the class of bounded surfaces that are diffeomorphic to an open
  subset of a compact surface; typical examples are the finite cylinder and
  the graph of a $C^2$ function.
\end{remark}

The proof of Theorem~\ref{thm:main} is subdivided into four steps. In
Subsection~\ref{susec:step1}, we show that for any $\varepsilon \in I_{\delta
+}$ and any $\vv_{\varepsilon} \in H^1( \Om_{\varepsilon}, \jM)$
the equality $\mathcal{G}_{\varepsilon} \left( \vv_{\varepsilon} \right)
=\mathcal{E}_{\calN}^{\varepsilon} \left( \vv_{\varepsilon} \circ
\psi_{\varepsilon} \right)$ holds, where $\psi_{\varepsilon}$ stands for the
diffeomorphism of $\calN$ onto $\Om_{\varepsilon}$ given by
\eqref{eq:diffeomorphism}. In Subsection \ref{subsec:equicoercivityffprime},
we show that the family $\left( \mathcal{E}_{\calN}^{\varepsilon}
\right)_{\varepsilon \in I_{\delta+}}$ is equicoercive in the weak topology
of $H^1( \calN, \jM)$. The identification of the $\Gamma$-limit
$\mathcal{E}_{\calN}$ is given in Subsection~\ref{subsec:identificationEN}.
Finally, in Section~\ref{sec:applicationsMicromag}, we consider several
instances of our main result and remark on the importance they can have in the
micromagnetic community, both as a source of new mathematical problems and new
models of interest for the physical community.

\section{The curved thin-film limit: proof of Theorem
\ref{thm:main}}\label{sec:proofmain}

\subsection{The equivalence of $\mathcal{G}_{\varepsilon}$ and
$\mathcal{E}_{\calN}^{\varepsilon}$}\label{susec:step1}In this section, we
prove the first part of Theorem~\ref{thm:main}, namely that once introduced,
for any $\varepsilon \in I_{\delta}$, the diffeomorphism of $\calN$ onto
$\Om_{\varepsilon}$ given by $\psi_{\varepsilon} : (\xi, \sss)
\in \calN \mapsto \xi + \varepsilon \sss \tmmathbf{n}_{\N} (\xi) \in
\Om_{\varepsilon}$, one has $\mathcal{G}_{\varepsilon} \left(
\vv_{\varepsilon} \right) \eqs \mathcal{E}_{\calN}^{\varepsilon} \left(
\vv_{\varepsilon} \circ \psi_{\varepsilon} \right)$ for any $\vv_{\varepsilon}
\in H^1( \Om_{\varepsilon}, \jM)$ and, therefore,
$\vv_{\varepsilon}$ minimizes $\mathcal{G}_{\varepsilon}$ if and only if
$\uu_{\varepsilon} (\xi, \sss) \assign \vv_{\varepsilon} \left(
\psi_{\varepsilon} (\xi, \sss) \right)$ minimizes
$\mathcal{E}_{\calN}^{\varepsilon}$. For that, we need the following result.

\begin{lemma}
  \label{eq:lemmametricfactor}Let $0 < \varepsilon < \delta$ and $s \in I$.
  Set $\N_{\varepsilon \sss} \assign \left\{ x \in \Om_{\varepsilon} \of d
  \left( x, \N \right) = \varepsilon \sss \right\}$. The following assertions
  hold:
  \begin{enumerateroman}
    \item The metric factor which relates the volume form on $\N_{\varepsilon
    \sss}$ to the volume form on $\N$ is given by
    \begin{equation}
      \sqrt{\mathfrak{g}_{\varepsilon} (\xi, \sss)} \assign
      \Pi_{i = 1}^{n - 1} \left( 1 + \varepsilon \sss \kappa_i (\xi) \right) .
      \label{eq:metrcoeffvolume}
    \end{equation}
    In particular, when $n = 3$, one has
    \begin{equation}
      \sqrt{\mathfrak{g}_{\varepsilon} (\xi, \sss)} \assign
      \left| 1 + 2 \left( \varepsilon \sss \right) H (\xi) + \left(
      \varepsilon \sss \right)^2 G (\xi) \right|,
    \end{equation}
    where $H (\xi)$ and $G (\xi)$ are, respectively, the mean and Gaussian
    curvature at $\xi \in \N$.
    
    \item The metric coefficients $\mathfrak{h}_{1, \varepsilon} \left( \xi,
    \sss \right), \mathfrak{h}_{2, \varepsilon} (\xi, \sss),
    \ldots, \mathfrak{h}_{n - 1, \varepsilon} (\xi, \sss)$ \
    which connect the tangential gradient on $\N_{\varepsilon \sss}$ to the
    tangential gradient on $\N$ are given by
    \begin{equation}
      \mathfrak{h}_{i, \varepsilon} (\xi, \sss) \assign
      \frac{1}{1 + \varepsilon \sss \kappa_i (\xi)} \quad (i = 1, 2, \ldots, n
      - 1) . \label{eq:metrcoeffgradient}
    \end{equation}
    In other words, if $\uu_{\varepsilon} (\xi, \sss) \assign
    \vv_{\varepsilon} \left( \psi_{\varepsilon} (\xi, \sss)
    \right)$, then for every $i = 1, \ldots, n - 1$, there holds
    \[ \left( \partial_{\tmt_i (\xi)} \vv_{\varepsilon} \right) \circ
       \psi_{\varepsilon} (\xi, \sss) \eqs \mathfrak{h}_{i,
       \varepsilon} (\xi, \sss) \partial_{\tmt_i (\xi)}
       \uu_{\varepsilon} (\xi, \sss) . \]
    \item The metric coefficient $\mathfrak{h}_n (\xi, \sss)$,
    which connects the normal derivative on $\N_{\varepsilon \sss}$ to the
    $\sss$-derivative on $I$ is given by $\mathfrak{h}_{n, \varepsilon} \left(
    \xi, \sss \right) = 1 / \varepsilon$. In other words, if
    $\uu_{\varepsilon} (\xi, \sss) \assign \vv_{\varepsilon}
    \left( \psi_{\varepsilon} (\xi, \sss) \right)$ then
    \begin{equation}
      \left( \partial_{\tmmathbf{n}_{\N} (\xi)} \vv_{\varepsilon} \right)
      \circ \psi_{\varepsilon} (\xi, \sss) \assign
      \frac{1}{\varepsilon} \partial_s \uu_{\varepsilon} \left( \xi, \sss
      \right) . \label{eq:metrcoeffgradient2n}
    \end{equation}
  \end{enumerateroman}
\end{lemma}

\begin{proof}
  We use the language of differential forms as it simplifies the argument.
  
  {\noindent}{\tmem{i.}} For every $\varepsilon \in I_{\delta}$ the
  $n$-dimensional thin domain $\Om_{\varepsilon}$ is diffeomorphic to the
  product manifold $\N \times I$ via the (positively oriented) map
  $\psi_{\varepsilon} : (\xi, s) \in \N \times I \mapsto \xi + \varepsilon
  \sss \tmmathbf{n}_{\N} (\xi) \in \Om_{\varepsilon}$. The tangent map $d
  \psi_{\varepsilon} (\xi, \sss)$ at the point $\left( \xi, \sss
  \right)$ is the linear map from $T_{(\xi, s)} \left( \N \times I \right)
  \approx T_{\xi} \N \oplus \RR$ into $\RR^m$ defined for every $\jmu \assign
  \left( \tmt, t \right) \in T_{\xi} \N \oplus \RR$ by
  \begin{equation}
    d \psi_{\varepsilon} (\xi, \sss) \jmu = \left( \tmt +
    \varepsilon \sss \partial_{\tmt} \tmmathbf{n}_{\N} (\xi) \right) +
    \varepsilon t\tmmathbf{n}_{\N} (\xi) . \label{eq:differentialpsi}
  \end{equation}
  Now, consider the orthonormal basis of $T_{\xi} \N \oplus \RR$ given by
  \begin{equation}
  \left( \jmu_1, \ldots, \jmu_n \right) = \left( \tmt_1 \oplus 0, \ldots,
  \tmt_{n - 1} \oplus 0, \tmmathbf{0} \oplus 1 \right) ,
  \end{equation}
  where the unit tangent
  vectors $\left( \tmt_i \right)_{i = 1}^n$ are along the principal directions
  of $\N$ at $\xi$, so that
  \begin{equation}
    \partial_{\tmt_i} \tmmathbf{n} (\xi) = \kappa_i (\xi) \tmt_i (\xi) \quad
    \text{for } i = 1, \ldots, n - 1,
  \end{equation}
  with $\kappa_i$ the $i$-th principal curvatures of $\N$ at $\xi$. We then
  have
  \begin{align}
    d \psi_{\varepsilon} (\xi, \sss) \jmu_i & = \left( 1 +
    \varepsilon \sss \kappa_i \right) \tmt_i \quad \text{for } i = 1, \ldots,
    n - 1,  \label{eq:differentialpsiVal}\\
    d \psi_{\varepsilon} (\xi, \sss) \jmu_n & =  \varepsilon
    \tmmathbf{n}_{\N} .  \label{eq:differentialpsiVal2}
  \end{align}
  We denote by $\mathd  x \assign d x_1 \wedge \cdots \wedge d x_n$ the volume
  form on $\RR^n$ and by $(d \mu_i)_{i = 1}^n$ the basis of $\left( T_{\xi} \N
  \oplus \RR \right)^{\ast}$ dual to $\left( \jmu_i \right)_{i = 1}^n$. Given
  that both $\mathd  x$ and $d \mu_1 \wedge \cdots \wedge d \mu_n$ are forms
  of maximal degree $n$, there exists a function $\lambda \left( \xi, \sss
  \right)$ such that $(\psi_{\varepsilon}^{\ast} \mathd  x) \left( \xi, \sss
  \right) = \lambda (\xi, \sss) (d \mu_1 \wedge \cdots \wedge d
  \mu_n)$. A simple computation then gives
  \begin{align}
    \lambda & =  (\psi_{\varepsilon}^{\ast} \mathd  x) \left( \jmu_1, \ldots,
    \jmu_n \right) \\
    & =  (\mathd  x) \left( \left( 1 + \varepsilon \sss \kappa_1 \right)
    \tmt_1, \ldots, \left( 1 + \varepsilon \sss \kappa_{n - 1} \right) \tmt_{n
    - 1}, \varepsilon \tmmathbf{n}_{\N} \right) \\
    & \eqs  \varepsilon \Pi_{i = 1}^{n - 1} \left| 1 + \varepsilon \sss
    \kappa_{n - 1} \right|, 
  \end{align}
  and we can always assume that $\varepsilon$ is sufficiently small so that
  \begin{equation}
  \varepsilon \Pi_{i = 1}^{n - 1} | 1 + \varepsilon s \kappa_{n - 1} | =
  \varepsilon \Pi_{i = 1}^{n - 1} (1 + \varepsilon s \kappa_{n - 1}).
  \end{equation}
 Overall we get that
  \begin{equation}
    \frac{1}{\varepsilon} (\psi_{\varepsilon}^{\ast} \mathd  x) \left( \xi,
    \sss \right) = \sqrt{\mathfrak{g}_{\varepsilon} (\xi, \sss)}
    (d \mu_1 \wedge \cdots \wedge d \mu_n)
  \end{equation}
  with $\sqrt{\mathfrak{g}_{\varepsilon} (\xi, \sss)}$ given by
  \eqref{eq:metrcoeffvolume}. This proves {\tmem{i}}.
  
  {\noindent}({\tmem{ii}}) The expression of the metric coefficients
  $\mathfrak{h}_{i, \varepsilon}$ is a simple application of chain rule and
  \eqref{eq:differentialpsiVal}. Indeed, for every $i = 1, \ldots, n - 1$we
  have
  \begin{equation}
    d \uu_{\varepsilon} (\xi, \sss) \left( \jmu_i \right) \eqs d 
    \vv_{\varepsilon} \left( \psi_{\varepsilon} (\xi, \sss)
    \right) d \psi_{\varepsilon} (\xi, \sss) \jmu_i \eqs \left( 1
    + \varepsilon \sss \kappa_i \right) d  \vv_{\varepsilon} \left(
    \psi_{\varepsilon} (\xi, \sss) \right) \tmt_i
  \end{equation}
  from which the relations in \eqref{eq:metrcoeffgradient} follow.
  
  {\noindent}({\tmem{iii}}) To compute the metric coefficients
  $\mathfrak{h}_{n, \varepsilon}$, we proceed as in {\tmem{ii}} using
  \eqref{eq:differentialpsiVal2}. We have
  \begin{equation}
    d \uu_{\varepsilon} (\xi, \sss) \left( \jmu_n \right) \eqs
    \varepsilon d  \vv_{\varepsilon} (\psi_{\varepsilon} (\xi, s))
    \tmmathbf{n}_{\N}
  \end{equation}
  from which \eqref{eq:metrcoeffgradient2n} follows.
\end{proof}
The previous Lemma~\ref{eq:lemmametricfactor} allows us to fix the domain of
integration and gives an equivalent representation of
$\mathcal{G}_{\varepsilon}$ from an integral functional defined on
$\Om_{\varepsilon}$ to an integral functional on $\calN = \N \times I$.

By coarea formula and \eqref{eq:metrcoeffvolume} from
Lemma~\ref{eq:lemmametricfactor} we get that
\begin{align}
  \mathcal{G}_{\varepsilon} \left( \vv_{\varepsilon} \right) & \eqs 
  \frac{1}{2 \varepsilon} \int_{\Om_{\varepsilon}} \left| D \vv_{\varepsilon}
  (x) + \Aop \left( \vv_{\varepsilon} (x) \right) \right|^2 \mathd x \\
  & \eqs
   \frac{1}{2 \varepsilon} \int_I \int_{\N_{\varepsilon \sss}} \left| D
  \vv_{\varepsilon} (y) + \Aop \left( {\vv_{\varepsilon}}  (y) \right)
  \right|^2 \mathd \mathcal{H}^{n - 1} (y) \mathd \sss \nonumber\\
  & \eqs \frac{1}{2} \int_{\calN} \left| D \vv_{\varepsilon} \circ
  \psi_{\varepsilon} (\xi, \sss) + \Aop \left( \left(
  \vv_{\varepsilon} \circ \psi_{\varepsilon} \right) (\xi, \sss)
  \right) \right|^2  \sqrt{\mathfrak{g}_{\varepsilon} \left( \xi, \sss
  \right)} \mathd \mathcal{H}^{n - 1} (\xi) \mathd \sss .  \label{eq:jfirsteq}
\end{align}
Next, projecting the gradient onto the orthonormal (moving) frame induced by $\N$, i.e., on
\begin{equation}
\left(
\tmt_1 (\xi), \tmt_2 (\xi), \ldots, \tmt_{n - 1} (\xi), \tmmathbf{n}_{\N}
(\xi) \right)\,
\end{equation} 
we get that for any $x \in \Om_{\varepsilon}$
there holds
\begin{align}
 \left| D \vv_{\varepsilon} (x) + \Aop \left( \vv_{\varepsilon} (x) \right)
   \right|^2 & =  \sum_{i = 1}^{n - 1} \left| \partial_{\tmt_i (\xi)}
   \vv_{\varepsilon} (x) + \Aop \left( \vv_{\varepsilon} (x) \right) \bp
   \tmt_i (\xi) \right|^2  \\
   &\qquad \qquad + \left| \partial_{\tmmathbf{n}_{\N} (\xi)}
   \vv_{\varepsilon} (x) + \Aop \left( \vv_{\varepsilon} (x) \right) \bp
   \tmmathbf{n}_{\N} (\xi) \right|^2 \, ,
\end{align}
with $\xi = \pi_{\N} (x)$ the nearest point projection of $x$ on $\N$ defined
by \eqref{eq:npprojection}. Therefore, from \eqref{eq:metrcoeffgradient} and
\eqref{eq:metrcoeffgradient2n} we have that
\begin{align}
  \left| D \vv_{\varepsilon} \circ \psi_{\varepsilon} (\xi, \sss)
  + \Aop \left( \left( \vv_{\varepsilon} \circ \psi_{\varepsilon} \right)
  (\xi, \sss) \right) \right|^2 & \eqs \sum_{i = 1}^{n - 1}
  \left| \mathfrak{h}_{i, \varepsilon} (\xi, \sss)
  \partial_{\tmt_i (\xi)} \uu_{\varepsilon} (\xi, \sss) + \Aop
  \left( \uu_{\varepsilon} (\xi, \sss) \right) \bp \tmt_i (\xi)
  \right|^2 \nonumber\\
  &  \qquad + \left| \frac{1}{\varepsilon} \partial_{\sss}
  \uu_{\varepsilon} (\xi, \sss) + \Aop \left( \uu_{\varepsilon}
  (\xi, \sss) \right) \bp \tmmathbf{n}_{\N} (\xi) \right|^2 . 
  \label{eq:jseceq}
\end{align}
Combining \eqref{eq:jfirsteq} and \eqref{eq:jseceq}, we infer the equality of
$\mathcal{G}_{\varepsilon} \left( \vv_{\varepsilon} \right)$ and
$\mathcal{E}_{\N}^{\varepsilon} (\tmmathbf{u}_{\varepsilon})$. Note that the
previous computation also shows that $\vv_{\varepsilon} \in H^1(
\Om_{\varepsilon}, \jM)$ if, and only if, $\tmmathbf{u}_{\varepsilon}
\in H^1( \calN, \jM)$. Finally, since the superposition operator
$\tmmathbf{v}_{\varepsilon} \in H^1( \Om_{\varepsilon}, \jM)
\mapsto (\tmmathbf{v}_{\varepsilon} \circ \psi_{\varepsilon}) \in H^1(
\calN, \jM)$ is surjective, we get that
\begin{equation}
  \inf_{\vv_{\varepsilon} \in H^1( \Om_{\varepsilon}, \jM)}
  \mathcal{G}_{\varepsilon} \left( \vv_{\varepsilon} \right) \; \eqs
  \inf_{\tmmathbf{u}_{\varepsilon} \in H^1( \calN, \jM)}
  \mathcal{E}_{\N}^{\varepsilon} \left( \uu_{\varepsilon} \right) .
  \label{eq:minGLonM}
\end{equation}
This concludes the proof of the first part of Theorem~\ref{thm:main}.

\subsection{Compactness}\label{subsec:equicoercivityffprime}We now show that
the family $( \mathcal{E}_{\calN}^{\varepsilon})_{\varepsilon \in
I_{\delta+}}$ is equicoercive in the weak topology of $H^1( \calN, \jM)$. This means, by definition, that there exists a nonempty and weakly
compact set $K \subseteq H^1( \calN, \jM)$ such that
\begin{equation}
  \min_{H^1( \calN, \jM)} \mathcal{E}_{\calN}^{\varepsilon} =
  \min_K \mathcal{E}_{\calN}^{\varepsilon} \quad \forall \varepsilon \in
  I_{\delta+} . \label{eq:equicoerciveE}
\end{equation}
The equicoercivity of $( \mathcal{E}_{\calN}^{\varepsilon}
)_{\varepsilon \in I_{\delta+}}$ will assure that we can rely on the
fundamental theorem of $\Gamma$-convergence concerning the variational
convergence of minimum problems (see, e.g., {\cite{dal1993introduction}}).
To show \eqref{eq:equicoerciveE}, we observe that for any constant in space
$\sigma \in H^1( \calN, \jM)$ we have
\begin{equation}
  \min_{\tmmathbf{u} \in H^1( \calN, \jM)}
  \mathcal{E}_{\calN}^{\varepsilon} ( \uu) \leqslant
  \mathcal{E}_{\calN}^{\varepsilon} (\sigma) = \frac{1}{2} \int_{\calN} \left|
  \Aop (\sigma) \right|^2 \sqrt{\mathfrak{g}_{\varepsilon} \left( \xi, \sss
  \right)} \mathd \xi \mathd \sss .
\end{equation}
Taking into account \eqref{eqs:lowbounds} and \eqref{eq:cALipnew}, we end up
with
\begin{equation}
  \min_{\tmmathbf{u} \in H^1( \calN, \jM)}
  \mathcal{E}_{\calN}^{\varepsilon} ( \uu) \; \leqslant \;
  \frac{c_{\Aop}^2}{2} \int_{\calN}   \sqrt{\mathfrak{g}_{\varepsilon} \left(
  \xi, \sss \right)} \mathd \xi \mathd \sss \; \leqslant \; c_{\Aop}^2 \cdot
  c_{\calN} \cdot \left| \N \right|,
\end{equation}
Therefore, for every $\varepsilon \in I_{\delta+}$, the minimizers of $\left(
\mathcal{E}_{\calN}^{\varepsilon} \right)_{\varepsilon \in I_{\delta+}}$ are
in
\begin{equation}
K( \calN, \jM) \assign \lcup_{\varepsilon \in I_{\delta+}}
   \left\{ \uu \in H^1( \calN, \jM) \of
   \mathcal{E}_{\calN}^{\varepsilon} ( \uu) \leqslant c_{\Aop}^2
   \cdot c_{\calN} \cdot \left| \N \right| \right\} .
   \end{equation}
Also, since the principal curvatures $\kappa_1, \ldots, \kappa_{n - 1}$ are
bounded on $\N$, whenever the radius $\delta \in \RR_+$ of the tubular
neighborhood $\Om_{\delta}$ is sufficiently small, we have that $\inf_{\left(
\xi, \sss \right) \in \calN} \mathfrak{h}_{i, \varepsilon} \left( \xi, \sss
\right) \geqslant c_{\calN}$ for every $\varepsilon \in I_{\delta+}$
({\tmabbr{cf.}} \eqref{eqs:lowbounds}). Therefore, from Young inequality, we
get
\begin{equation}
  \left| \mathfrak{h}_{i, \varepsilon} (\xi, \sss)
  \partial_{\tmt_i (\xi)} \uu_{\varepsilon} (\xi, \sss) + \Aop
  \left( \uu_{\varepsilon} (\xi, \sss) \right) \bp \tmt_i (\xi)
  \right|^2  \geqslant  \frac{1}{2 c_{\calN}^2} \left| \partial_{\tmt_i
  (\xi)} \uu_{\varepsilon} (\xi, \sss) \right|^2 - \left| \Aop
  \left( \uu_{\varepsilon} (\xi, \sss) \right) \bp \tmt_i (\xi)
  \right|^2,  \label{eq:lbcompact1}
\end{equation}
and
\begin{equation}
  \left| \frac{1}{\varepsilon} \partial_{\sss} \uu_{\varepsilon} \left( \xi,
  \sss \right) + \Aop \left( \uu_{\varepsilon} (\xi, \sss)
  \right) \bp \tmmathbf{n}_{\N} (\xi) \right|^2  \geqslant  \frac{1}{2
  \varepsilon^2} \left| \partial_{\sss} \uu_{\varepsilon} \left( \xi, \sss
  \right) \right|^2 - \left| \Aop \left( \uu_{\varepsilon} \left( \xi, \sss
  \right) \right) \bp \tmmathbf{n}_{\N} (\xi) \right|^2 .
  \label{eq:lbcompact2}
\end{equation}
Therefore
\begin{align}
   \sum_{i = 1}^{n - 1} \frac{1}{2 c_{\calN}^2} \left| \partial_{\tmt_i}
  \uu_{\varepsilon} \right|^2  + \frac{1}{2 \varepsilon^2} \left|
  \partial_{\sss} \uu_{\varepsilon} \right|^2 & \leqslant  + c_{\calN}
  \sum_{i = 1}^{n - 1} \left| \mathfrak{h}_{i, \varepsilon} \partial_{\tmt_i}
  \uu_{\varepsilon} + \Aop \left( \uu_{\varepsilon} \right) \bp \tmt_i
  \right|^2  \sqrt{\mathfrak{g}_{\varepsilon}}  \nonumber\\
  &   \hspace{4em} + c_{\calN} \left| \frac{1}{\varepsilon} \partial_{\sss}
  \uu_{\varepsilon} + \Aop \left( \uu_{\varepsilon} \right) \bp
  \tmmathbf{n}_{\N} \right|^2  \sqrt{\mathfrak{g}_{\varepsilon}} + c_{\calN}
  c_{\Aop}^2 .
\end{align}
It follows that if $\tmmathbf{u} \in K( \calN, \jM)$, then
\begin{align}
  \frac{1}{2 c_{\calN}^2} \int_{\calN} \left| \grad_{\xi} \uu_{\varepsilon}
  (\xi, \sss) \right|^2 \mathd \xi \mathd \sss + \frac{1}{2
  \varepsilon^2} \int_{\calN} \left| \partial_{\sss} \uu_{\varepsilon} \left(
  \xi, \sss \right) \right|^2 \mathd \xi \mathd \sss & \leqslant  2 c_{\calN}
  \mathcal{E}_{\calN}^{\varepsilon} ( \uu) + c_{\calN} c_{\Aop}^2
  \nonumber\\
  & \leqslant  c_{\Aop}^2 \left( 2 c_{\calN}^2 \left| \N \right| + c_{\calN}
  \right), 
\end{align}
and, in particular,
\begin{equation}
  \left\| \uu \right\|_{H^1( \calN, \jM)}^2 \leqslant c_{\Aop,
  \N}^2
\end{equation}
for some positive constant $c_{\Aop, \N}$ that does not depend on
$\varepsilon$. In other words, the set $K( \calN, \jM)$ is
contained in the bounded subset $H^1_b ( \calN, \jM)$ of $H^1( \calN, \RR^m)$ given by the intersection of $H^1( \calN,
\jM)$ with the ball of $H^1( \calN, \RR^m)$ centered at
the origin and of radius $c_{\Aop, \N}$. Thus, for any $\varepsilon \in
I_{\delta+}$
\begin{equation}
  \min_{\uu \in H^1( \calN, \jM)}
  \mathcal{E}_{\calN}^{\varepsilon} ( \uu) = \min_{\uu \in H^1_b( \calN, \jM)} \mathcal{E}_{\calN}^{\varepsilon} ( \uu) .
\end{equation}
To prove that $H^1_b ( \calN, \jM)$ is weakly compact, it is
sufficient to show that the set $H^1_b ( \calN, \jM)$ is weakly
closed. To this end, we simply observe that if $\left( \uu_n \right)_{n \in
\NN}$ is a sequence in $H^1_b ( \calN, \jM)$ such that $\uu_n
\rightharpoonup \uu_0$ weakly in $H^1( \calN, \RR^m)$, by
Rellich-Kondrachov theorem, $\uu_n \rightharpoonup \uu_0$ strongly in $L^2( \calN, \RR^m)$, and therefore, up to the extraction of an a.e.
pointwise converging subsequence, we get that the $\uu_0$ still takes values
in $\jM$. Indeed, the distance function from the boundary is a continuous
function and, therefore, if $d_{\jM} \left( \uu_n \right) = 0$, then also
$d_{\jM} (\uu_0 ) = 0$. Thus $\uu_0 (\xi, \sss) \in
\jM$ for {\tmabbr{a.e.}} $(\xi, \sss) \in \calN$ and this
concludes the equicoerciveness proof.

\subsection{The identification of the
$\Gamma$-limit}\label{subsec:identificationEN}In this section, we compute
$\mathcal{E}_{\calN} \assign \Gamma \text{-} \lim_{\varepsilon \rightarrow 0}
\mathcal{E}_{\calN}^{\varepsilon}$. We first establish a compactness result
({\tmabbr{cf.~}}Proposition~\ref{prop:forGammaliminf}). Afterward, in
Proposition~\ref{propo:GammalimiEexch}, we prove the $\Gamma$-limsup
inequality, i.e., the existence of a recovery sequence, and the
$\Gamma$-$\liminf$ inequality for $\left( \mathcal{E}_{\calN}^{\varepsilon}
\right)_{\varepsilon \in I_{\delta+}}$, i.e., that for any family $(
\uu_{\varepsilon} )_{\varepsilon \in I_{\delta+}}$ weakly convergent to
some $\uu_0 \in H^1( \calN, \jM)$ we have $\mathcal{E}_{\calN}
(\uu_0 ) \leqslant \limid{_{\varepsilon \rightarrow 0}
\mathcal{E}_{\calN}^{\varepsilon} ( \uu_{\varepsilon} )}$. With no
loss of generality, for the computation of $\Gamma$-$\liminf$ we can assume
that $\limid{_{\varepsilon \rightarrow 0} \mathcal{E}_{\calN}^{\varepsilon}
( \uu_{\varepsilon})} < + \infty$.
\begin{proposition}
  \label{prop:forGammaliminf}Assume that $\left( \uu_{\varepsilon}
  \right)_{\varepsilon \in I_{\delta+}}$ is a family in $H^1( \calN,
  \jM)$ such that
  \[ \limid{_{\varepsilon \rightarrow 0} \mathcal{E}_{\calN}^{\varepsilon}
     \left( \uu_{\varepsilon} \right)} < + \infty . \]
  Then there exist elements $\uu_0 \in H^1( \calN, \jM)$ and
  $\tmmathbf{d}_0 \in L^2( \calN, \RR^m)$ such that
  \begin{align}
    \uu_{\varepsilon} & \rightarrow \uu_0  \;\;  \text{weakly in } H^1(
    \calN, \jM),  \label{propgamliminf1}\\
    \frac{1}{\varepsilon} \partial_{\sss} \uu_{\varepsilon} \left( \xi, \sss
    \right) & \rightarrow  \tmmathbf{d}_0  \;\; \text{weakly in } L^2 (
    \calN, \RR^m) .  \label{propgamliminf2}
  \end{align}
  Moreover, $\uu_0$ is $0$-homogeneous along the normal to $\N$, i.e., it is
  an element $\uu_0$ of the form
  \begin{equation}
    \uu_0 (\xi, \sss) = \widetilde{\uu}_0 (\xi), \quad \text{for
    a.e. } (\xi, \sss) \in \calN \label{eq:tensorformu0}
  \end{equation}
  for some $\widetilde{\uu}_0 \in H^1( \N, \jM)$, and
  $\tmmathbf{d}_0$ is tangent to $\jM$ at $\uu_0$, i.e.,
  \begin{equation}
    \tmmathbf{d}_0 (\xi, \sss) \cdot \tmmathbf{n}_{\jM} (\tmmathbf{u}_0 (\xi))
    = 0 \quad \text{for a.e. } (\xi, \sss) \in \calN .
    \label{propgamliminf4}
  \end{equation}
\end{proposition}

\begin{proof}
  Using \eqref{eqs:lowbounds}, \eqref{eq:rewriteG1}, and
  \eqref{eq:lbcompact2}, we obtain
  \begin{align}
    + \infty \; > \; \limi{\varepsilon \rightarrow 0}
    \mathcal{E}_{\calN}^{\varepsilon} \left( \uu_{\varepsilon} \right) &
    \geqslant  \limi{\varepsilon \rightarrow 0}  \frac{1}{2}
    {\int_{\calN}}_{} \left| \frac{1}{\varepsilon} \partial_{\sss}
    \uu_{\varepsilon} (\xi, \sss) + \Aop \left( \uu_{\varepsilon}
    (\xi, \sss) \right) \bp \tmmathbf{n}_{\N} (\xi) \right|^2 
    \sqrt{\mathfrak{g}_{\varepsilon} (\xi, \sss)} \mathd \xi
    \mathd \sss, \nonumber\\
    & \geqslant   \limi{\varepsilon \rightarrow 0}  \frac{c^{-
    1}_{\calN}}{4 \varepsilon^2} {\int_{\calN}}_{} \left| \partial_{\sss}
    \uu_{\varepsilon} (\xi, \sss) \right|^2 \mathd \xi \mathd
    \sss  - \frac{1}{2} c^{- 1}_{\calN} c_{\Aop}^2 .  \label{eq:tempk1}
  \end{align}
  From the previous estimate, \eqref{propgamliminf1} and
  \eqref{propgamliminf2} immediately follows. Also, by the lower
  semicontinuity of the norm, we get that $0 = \lim_{\varepsilon \rightarrow
  0} \left\| \partial_{\sss} \uu_{\varepsilon} \right\|_{L^2(\calN)} = \left\| \partial_{\sss} \uu_0 \right\|_{L^2(\calN
  )}$. Since we also have $\partial_{\sss} \uu_{\varepsilon}
  \rightharpoonup \partial_{\sss} \uu_0$ in $\mathcal{D}'( \calN)$, overall, we get that
  \begin{equation}
    \partial_{\sss} \uu_{\varepsilon} \rightarrow \partial_{\sss} \uu_0 \left(
    \xi, \sss \right)  \text{ strongly in } L^2 ( \calN, \RR^m),
    \quad \partial_{\sss} \uu_0 (\xi, \sss) = 0 \text{ a.e. in }
    \calN . \label{eq:tempk3}
  \end{equation}
  This proves \eqref{eq:tensorformu0}.
  
  To show \eqref{propgamliminf4}, we denote by $d_{\jM}$ the signed distance
  from $\jM$, which in our hypotheses on $\jM$ is such that $\grad d_{\jM}
  (\sigma) =\tmmathbf{n}_{\jM} (\sigma)$. We then have $0 = \partial_{\sss}
  \left( d_{\jM} \left( \uu_{\varepsilon} \right) \right) =\tmmathbf{n}_{\jM}
  \left( \uu_{\varepsilon} \right) \cdot \partial_{\sss} \uu_{\varepsilon}$
  with $\varepsilon^{- 1} \partial_{\sss} \uu_{\varepsilon} \rightharpoonup
  \tmmathbf{d}_0$ weakly in $L^2(\calN, \RR^m)$ and
  $\tmmathbf{n}_{\jM} \left( \uu_{\varepsilon} \right) \rightarrow
  \tmmathbf{n}_{\jM} (\uu_0 )$ strongly in $L^2 \left( \jM, \RR^m
  \right)$. This shows \eqref{propgamliminf4}.
\end{proof}

\begin{proposition}
  \label{propo:GammalimiEexch}The family $(
  \mathcal{E}_{\calN}^{\varepsilon} )_{\varepsilon \in I_{\delta+}}$ of
  exchange energy on $\calN$, $\Gamma$-converges, with respect to the weak
  topology of $H^1( \calN, \jM)$, to the functional
  {\tmem{\eqref{eq:gleF}}}, i.e., to the functional $\mathcal{E}_{\calN}$
  defined on $H^1( \calN, \jM)$ by
  \begin{equation}
    \mathcal{E}_{\calN} ( \uu) = \int_{\N} \left|  \grad_{\xi} 
    \uu (\xi) + \Aop_{\xi} (\uu(\xi)) \right|^2 + \left( \Aop
    (\uu(\xi)) \tmmathbf{n}_{\N} (\xi) \cdot \tmmathbf{n}_{\jM}
    (\uu(\xi)) \right)^2 \mathd \xi \label{eq:GammaliEprime0}
  \end{equation}
  if $\uu$ does not depend on the $\sss$-variable, and $\mathcal{E}_{\calN}
  ( \uu) = + \infty$ otherwise.
\end{proposition}

\begin{proof}
  We split the proof into two steps. We start by addressing the
  $\Gamma$-$\liminf$ inequality for $\left( \mathcal{E}_{\calN}^{\varepsilon}
  \right)_{\varepsilon \in I_{\delta+}}$.
  
  {\noindent}{\tmstrong{Step~1~[$\Gamma$-$\text{liminf}$ inequality].}} By
  Proposition~\ref{prop:forGammaliminf}, we have that for the identification
  of the $\Gamma$-limit of $\left( \mathcal{E}_{\calN}^{\varepsilon}
  \right)_{\varepsilon \in I_{\delta}}$ it is sufficient to consider those
  families of $H^1( \calN, \jM)$ functions that weakly converge
  to a $0$-homogeneous element $\uu_0 \in H^1( \calN, \jM)$,
  i.e., to an element $\uu_0$ of the form $\uu_0 (\xi, \sss) =
  \chi_I \left( \sss \right) \widetilde{\uu}_0 (\xi)$ for some
  $\widetilde{\uu}_0 \in H^1( \N, \jM)$, i.e., not depending on
  the $\sss$ variable. In the following, with a slight abuse of notation, we
  write $\uu_0 (\xi)$ instead of $\widetilde{\uu}_0 (\xi)$.
  
  Taking into account the lower semicontinuity of the norm, for any
  $\uu_{\varepsilon} \rightharpoonup \uu_0$ in $H^1( \calN, \jM)$, with $\uu_0$ of the type \eqref{eq:tensorformu0}, we get
  ({\tmabbr{cf.}}~\eqref{eq:rewriteG1})
  \begin{align}
    \liminf_{\varepsilon \rightarrow 0} \mathcal{E}_{\calN}^{\varepsilon}
    \left( \uu_{\varepsilon} \right) & \geqslant  \sum_{i = 1}^{n - 1}
    \int_{\N} \left| \partial_{\tmt_i (\xi)} \uu_0 (\xi) + \Aop \left( \uu_0
    (\xi) \right) \bp \tmt_i (\xi) \right|^2 \mathd \xi \nonumber\\
    &   \hspace{4em} + \frac{1}{2} {\int_{\N \times I}}_{} \left|
    \tmmathbf{d}_0 (\xi, \sss) + \Aop \left( \uu_0 (\xi) \right) \bp
    \tmmathbf{n}_{\N} (\xi) \right|^2 \mathd \xi \mathd \sss . \nonumber
  \end{align}
  To shorten the notation, it is convenient to introduce the following
  notation. For every $\sigma \in \jM$ we denote by $\Aop_{\xi} (\sigma)$ the
  restriction of $\Aop (\sigma)$ to the tangent space of $\N$ at $\xi$
  represented by the matrix $\Aop_{\xi} (\sigma) = \left( \Aop (\sigma) \bp
  \tmt_1 (\xi), \ldots, \Aop (\sigma) \bp \tmt_{n - 1} (\xi) \right) \in
  \RR^{(n - 1) \times m}$. The previous estimate then reads as
  \begin{equation}
    \liminf_{\varepsilon \rightarrow 0} \mathcal{E}_{\calN}^{\varepsilon}
    \left( \uu_{\varepsilon} \right) \geqslant \int_{\N} \left|  \grad_{\xi} 
    \uu_0 (\xi) + \Aop_{\xi} \left( \uu_0 (\xi) \right) \right|^2 \mathd \xi +
    \Rop [\tmmathbf{d}_0] \label{eq:estimateRd0}
  \end{equation}
  with
  \begin{equation}
    \Rop [\tmmathbf{d}_0] \assign \frac{1}{2} {\int_{\N \times I}}_{} \left|
    \tmmathbf{d}_0 (\xi, \sss) + \Aop \left( \uu_0 (\xi) \right) \bp
    \tmmathbf{n}_{\N} (\xi) \right|^2 \mathd \xi \mathd \sss .
  \end{equation}
  Next, we minimize (pointwise) the integrand of $\Rop [\tmmathbf{d}_0]$
  {\tmem{under the constraint}} that $\tmmathbf{d}_0 (\xi, \sss) \cdot
  \tmmathbf{n}_{\jM} (\tmmathbf{u}_0 (\xi)) = 0$, i.e., we look for
  \begin{equation}
    \min_{\tmmathbf{d}_0 \cdot \tmmathbf{n}_{\jM} (\tmmathbf{u}_0) = 0} \left|
    \tmmathbf{d}_0 + \Aop (\tmmathbf{u}_0) \tmmathbf{n}_{\N} \right|^2 .
  \end{equation}
  This leads to the Euler--Lagrange equation $\tmmathbf{d}_0 + \Aop
  (\tmmathbf{u}_0) \tmmathbf{n}_{\N} = \left( \Aop (\tmmathbf{u}_0)
  \tmmathbf{n}_{\N} \cdot \tmmathbf{n}_{\jM} (\tmmathbf{u}_0) \right)
  \tmmathbf{n}_{\jM} (\tmmathbf{u}_0)$ which, in particular, shows that for
  every $(\xi, \sss) \in \calN$ the minimal vector $\tmmathbf{d}_0 (\xi,
  \sss)$ does not depend on the $\sss$ variable and is given by
  \begin{equation}
    \tmmathbf{d}_0 (\xi) \assign \left( \left[ \tmmathbf{n}_{\jM}
    (\tmmathbf{u}_0) \otimes \tmmathbf{n}_{\jM} (\tmmathbf{u}_0) \right] -
    \mathsf{I} \right) \Aop (\tmmathbf{u}_0) \tmmathbf{n}_{\N},
    \label{eq:defd0}
  \end{equation}
  with $\mathsf{I}$ denoting the identity operator.
  
  The previous relation leads to the following expression of the minimal
  energy density
  \begin{equation}
    \left| \tmmathbf{d}_0 (\xi, s) + \Aop (\tmmathbf{u}_0 (\xi))
    \tmmathbf{n}_{\N} (\xi) \right|^2 \eqs \left( \left( \Aop^{\T}
    (\tmmathbf{u}_0 (\xi)) \tmmathbf{n}_{\jM} (\tmmathbf{u}_0 (\xi)) \right)
    \cdot \tmmathbf{n}_{\N} (\xi) \right)^2
  \end{equation}
  and, therefore, from \eqref{eq:estimateRd0}, we infer that
  \begin{equation}
    \liminf_{\varepsilon \rightarrow 0} \mathcal{E}_{\calN}^{\varepsilon}
    \left( \uu_{\varepsilon} \right) \geqslant \int_{\N} \left|  \grad_{\xi} 
    \uu_0 (\xi) + \Aop_{\xi} \left( \uu_0 (\xi) \right) \right|^2 \mathd \xi +
    \int_{\N} \left( \Aop \left( \uu_0 (\xi) \right) \tmmathbf{n}_{\N} (\xi)
    \cdot \tmmathbf{n}_{\jM} \left( \uu_0 (\xi) \right) \right)^2 \mathd \xi .
    \label{eq:gammaliminf}
  \end{equation}
  This concludes the proof of the $\Gamma$-$\liminf$ inequality.
  
  {\noindent}{\tmstrong{Step~2~[$\Gamma$-$\text{limsup}$ inequality].}} We now
  address the existence of a recovery sequence, i.e., that the lower bound
  \eqref{eq:gammaliminf} is indeed optimal. To this end, it is sufficient to
  show that if $\uu_0 \in H^1( \calN, \jM)$ is independent of the
  $\sss$-variable, then there exists a sequence $\left(
  \uu^{\ast}_{\varepsilon} \right)_{\varepsilon \in I_{\delta+}}$ of
  functions in $H^1( \calN, \jM)$ such that
  $\uu^{\ast}_{\varepsilon} \rightarrow \uu_0$ strongly in $H^1( \calN,
  \jM)$. Indeed, if this is the case, then
  \begin{equation}
    \lim_{\varepsilon \rightarrow 0} \mathcal{E}_{\calN}^{\varepsilon} \left(
    \uu^{\ast}_{\varepsilon} \right) =\mathcal{E}_{\calN} (\uu_0 )
    . \label{eq:gammalimsupE}
  \end{equation}
  For $\uu_0 \in H^1( \calN, \jM)$ and independent of the
  $\sss$-variable, we claim that
  \begin{equation}
    \uu^{\ast}_{\varepsilon} (\xi, \sss) \assign \pi_{\jM} \left(
    \uu_0 (\xi) + \varepsilon s\tmmathbf{d}_0 (\xi) \right)
    \label{eq:recsequencebuildfrom}
  \end{equation}
  with $\tmmathbf{d}_0$ defined by \eqref{eq:defd0}, gives the recovery
  sequence. Here $\pi_{\jM}$ is the nearest point projection on $\jM$ defined
  in \eqref{eq:npprojection} and, therefore $\uu^{\ast}_{\varepsilon} \left(
  \xi, \sss \right)$ is the projection on $\jM$ of a suitable (and small)
  perturbation of $\uu_0$ along the tangent space of $\jM$ at $\uu_0 (\xi)$.
  
  We point out that $\uu^{\ast}_{\varepsilon}$ is well-defined because for
  almost every $( \xi, \sss ) \in \calN$ we have that
  \begin{align}
    \left| d_{\jM} ( \uu_0 (\xi) + \varepsilon s\tmmathbf{d}_0 (\xi)) \right| & \eqs  \left| d ( \uu_0 (\xi) + \varepsilon
    s\tmmathbf{d}_0 (\xi), \jM) \right| \nonumber\\
    & \leqslant  \left| d \left( \uu_0 (\xi) + \varepsilon s\tmmathbf{d}_0
    (\xi), \uu_0 (\xi) \right) \right| \nonumber\\
    & \leqslant  \varepsilon | \tmmathbf{d}_0 (\xi) |, 
  \end{align}
  with $| \tmmathbf{d}_0 (\xi) |$ bounded, uniformly in $\xi \in \N$, by some
  constant which depends only on $c_{\Aop}$. Therefore, for $\varepsilon$
  sufficiently small, $\uu_0 (\xi) + \varepsilon s\tmmathbf{d}_0 (\xi)$ is in
  a tubular neighborhood of $\N$ where the unique nearest point projection
  $\pi$ is defined.
  
  To evaluate $D \uu_{\varepsilon}^{\ast}$ we proceed as follows. For any $y$
  in a tubular neighborhood $O_{\delta}$ of $\jM$ we have
  \begin{equation}
    y = \pi_{\jM} (y) + d_{\jM} (y) \left( \tmmathbf{n}_{\jM} \circ \pi_{\jM}
    \right) (y) .
  \end{equation}
  Therefore, if we set $\tilde{\tmmathbf{n}}_{\jM} (y) \assign \left(
  \tmmathbf{n}_{\jM} \circ \pi_{\jM} \right) (y)$, given that
  $\tilde{\tmmathbf{n}}_{\jM} = \tilde{\tmmathbf{n}}_{\jM} \circ \pi_{\jM}$,
  passing to the tangent maps in the previous relation, we have that
  $\mathsf{I} = D \pi_{\jM} (y) + \tilde{\tmmathbf{n}}_{\jM} (y) \otimes
  \tilde{\tmmathbf{n}}_{\jM} (y) + d_{\jM} (y) D \tilde{\tmmathbf{n}}_{\jM}
  (y) D \pi_{\jM} (y)$ with $\mathsf{I}$ the identity operator. Hence, the
  following relation holds
  \begin{equation}
    \mathsf{I} - \tilde{\tmmathbf{n}}_{\jM} (y) \otimes
    \tilde{\tmmathbf{n}}_{\jM} (y)= \left[ \mathsf{I} + d_{\jM} (y) D
    \tilde{\tmmathbf{n}}_{\jM} (y) \right] \bp D \pi_{\jM} (y), 
  \end{equation}
  from which we get that
  \begin{equation}
    D \pi_{\jM} (y) \eqs \left[ \mathsf{I} + d_{\jM} (y) D
    \tilde{\tmmathbf{n}}_{\jM} (y) \right]^{- 1} \bp \left[ \mathsf{I} -
    \tilde{\tmmathbf{n}}_{\jM} (y) \otimes \tilde{\tmmathbf{n}}_{\jM} (y)
    \right] \label{eq:actsasidentitytan}
  \end{equation}
  which is continuous in $y \in O_{\delta}$ because of $d_{\jM} (y)$ being
  small. Therefore, with $\tmmathbf{\eta}_{\varepsilon} \left( \xi, \sss
  \right) \assign \uu_0 (\xi) + \varepsilon \sss \tmmathbf{d}_0 (\xi)$, we
  have that
  \begin{align}
    \partial_{\tmt_i (\xi)} \uu^{\ast}_{\varepsilon} (\xi, \sss)
    & \eqs  D \pi_{\jM} \left( \tmmathbf{\eta}_{\varepsilon} \left( \xi, \sss
    \right) \right) \bp \partial_{\tmt_i (\xi)} \tmmathbf{\eta}_{\varepsilon}
    (\xi, \sss), \\
    \partial_{\sss} \uu^{\ast}_{\varepsilon} (\xi, \sss) & \eqs 
    \varepsilon D \pi_{\jM} \left( \tmmathbf{\eta}_{\varepsilon} ( \xi,
    \sss ) \right) \bp \tmmathbf{d}_0 (\xi) . 
  \end{align}
  But then, since for every $\sigma \in \jM$ the map $D \pi_{\jM} (\sigma)$
  acts as the identity on the tangent space $T_{\sigma} \jM$ of $\jM$ at
  $\sigma$, taking into account \eqref{eq:actsasidentitytan}, we get the
  following relations
  \begin{alignat}{3}
    \uu^{\ast}_{\varepsilon} (\xi, \sss) & \rightarrow
    \uu_0 (\xi) \hspace{1.2em} && \text{ strongly in } L^2(\calN, \jM),  \\
    \grad_{\xi} \uu^{\ast}_{\varepsilon} (\xi, \sss) &
    \rightarrow  \grad_{\xi} \uu_0 (\xi) \hspace{1.2em} && \text{ strongly in }
    L^2(\calN, \left( T \jM)^{n - 1} \right),\\
    \partial_{\sss} \uu^{\ast}_{\varepsilon} (\xi, \sss) &
    \rightarrow  0 \hspace{1.2em} && \text{ strongly in } L^2(\calN, T
    \jM), \\
    \varepsilon^{- 1} \partial_{\sss} \uu^{\ast}_{\varepsilon} \left( \xi,
    \sss \right) & \rightarrow  \tmmathbf{d}_0 (\xi) \hspace{1.2em} && \text{
    strongly in } L^2(\calN, T \jM). 
  \end{alignat}
  In particular, we have that $\uu^{\ast}_{\varepsilon} \left( \xi, \sss
  \right) \rightarrow \uu_0 (\xi)$ strongly in $H^1( \calN, \jM)$
  and, therefore, \eqref{eq:gammalimsupE} holds.
\end{proof}

\subsection{Strong convergence of
minimizers}\label{subsec:laststepstrong}Assertion
\eqref{eq:firstordergammadevelopforF}, as well as the {\tmem{weak}}
convergence in $H^1( \calN, \jM)$ of a minimizing family for $(
\mathcal{E}_{\calN}^{\varepsilon} )_{\varepsilon \in I_{\delta+}}$ to a
minimum point of
$\mathcal{E}_{\calN}$, \ is nothing but the fundamental theorem of
$\Gamma$-convergence concerning the variational convergence of minimum
problems ({\tmabbr{cf.}}~{\cite{dal1993introduction}}), which holds under the
equicoercivity result we proved in Subsection
\ref{subsec:equicoercivityffprime}. Therefore, it remains to prove that the
convergence is (up to a subsequence) {\tmem{strong}} in $H^1( \calN, \jM)$.

To this end, we observe that, by assumptions, there exists $\uu_0 \in H^1
( \calN, \jM)$, not depending on the $\sss$-variable, such that
$\uu_{\varepsilon} \rightharpoonup \uu_0$ weakly in $H^1( \calN, \jM)$ and $\partial_{\sss} \uu_{\varepsilon} \rightarrow 0$ strongly in
$L^2(\calN, T \jM)$. By the $\Gamma$-liminf inequality, the
strong convergence of $\uu_{\varepsilon} \rightarrow \uu_0$ in $L^2(
\calN, \jM)$, and the minimality of $\uu_{\varepsilon}$, we get that
\begin{equation}
  \mathcal{E}_{\calN} (\uu_0 ) \leqslant \liminf_{\varepsilon
  \rightarrow 0} \mathcal{E}_{\calN}^{\varepsilon} \left( \uu_{\varepsilon}
  \right) \leqslant \limsup_{\varepsilon \rightarrow 0}
  \mathcal{E}_{\calN}^{\varepsilon} \left( \uu_{\varepsilon} \right) \leqslant
  \lim_{\varepsilon \rightarrow 0} \mathcal{E}_{\calN}^{\varepsilon} \left(
  \uu_{\varepsilon}^{\ast} \right) \eqs \mathcal{E}_{\calN} \left( \uu_0
  \right),
\end{equation}
where $\left( \uu_{\varepsilon}^{\ast} \right)_{\varepsilon \in I_{\delta+}}$
is the family built from $\uu_0$ as in \eqref{eq:recsequencebuildfrom}.
Therefore, we conclude that if $\left( \uu_{\varepsilon} \right)_{\varepsilon
\in I_{\delta+}}$ is a minimizing family associated with $\left(
\mathcal{E}_{\calN}^{\varepsilon} \right)_{\varepsilon \in I_{\delta+}}$,
then $\uu_{\varepsilon} \rightharpoonup \uu_0$ weakly in $H^1( \calN,
\jM)$ for some $\uu_0$ which does not depend on the $\sss$-variable
and, moreover, $\lim_{\varepsilon \rightarrow 0}
\mathcal{E}_{\calN}^{\varepsilon} \left( \uu_{\varepsilon} \right)
=\mathcal{E}_{\calN} (\uu_0 )$. But this last relation, in
expanded form, is equivalent to the relation
\begin{equation}
  \frac{1}{2} \int_{\calN} \left|  \grad_{\xi}  \uu_0 (\xi) + \Aop_{\xi}
  \left( \uu_0 (\xi) \right) \right|^2 \mathd \xi \mathd \sss \eqs
  \lim_{\varepsilon \rightarrow 0}  \frac{1}{2} \int_{\calN} \left|
  \grad_{\xi}  \uu_{\varepsilon} (\xi, \sss) + \Aop_{\xi} \left(
  \uu_0 (\xi) \right) \right|^2 \mathd \xi \mathd \sss .
  \label{eq:strongconvnorms}
\end{equation}
From \eqref{eq:strongconvnorms} we deduce the convergence of the norms
$\| \uu_{\varepsilon}\|_{H^1( \calN, \jM)}
\rightarrow \| \uu_0 \|_{H^1( \calN, \jM)}$ which
together with the weak convergence $\uu_{\varepsilon} \rightharpoonup \uu_0$
in $H^1( \calN, \jM)$, assures strong convergence in $H^1(\calN, \jM)$.

\section{Applications to Micromagnetics}\label{sec:applicationsMicromag}

Because of the growing interest in spintronics applications, magnetic
skyrmions are currently the focus of considerable research activity ranging
from mathematics to physics and materials science. These chiral structures can
be found in a wide range of experimental conditions. This section demonstrates
how our curved thin film analysis can account for various situations that
might arise when ferromagnetic crystals lack inversion symmetry, and the
Dzyaloshinskii-Moriya interaction (DMI) can twist the otherwise uniform
ferromagnetic state.

Generally speaking, due to material defects or anisotropy in saturation
magnetization~{\cite{Venkatesan_2004}}, the loss of the $\Stwo^2$-valued
constraint is possible. As we show below
({\tmabbr{cf.}}~\eqref{eq:1stcase2}), even in the planar setting, extra
contributions to those reported in {\cite{Davoli2022}} must be taken into
consideration every time the magnetization vector is not precisely
$\Stwo^2$-valued. A specific instance of our findings provides a model for
curved thin films in the presence of bulk DMI.

In addition, we analyze the curved thin-film limit associated with interfacial
DMI. This effect occurs when two ferromagnetic materials with different
crystal structures are separated by an interface.

Finally, we consider the case when the temperature of the ferromagnet is not
uniform and, as a result, the (temperature-dependent) saturation magnetization
cannot be assumed constant. We show how to use Theorem~\ref{thm:maingen} to
cover this situation.

\subsection{Isotropic Bulk DMI}\label{subsec:isotropicbulkDMI}Isotropic bulk
DMI arises when Dzyaloshinskii vectors $(\tmmathbf{d}_i)_{i = 1}^3$ in
\eqref{eq:H1asymDMI} are the elements of the standard basis
$(\tmmathbf{e}_i)_{i = 1}^3$ of $\RR^3$. In this circumstance, the
perturbation $\Aop$ takes the form $\Aop \left( \m \right) \eqs \kappa \left(
\bullet \, \times \m \right)$, $\kappa \in \RR$. Theorem 1 from
{\cite{Davoli2022}} is about this setting in the particular case where $N$ is
a planar surface and $M = \Stwo^2$. More generally, when $\jM, \N$ are
two-dimensional surfaces in $\RR^3$, the energy functional
\eqref{eq:mainenfunc} associated with the isotropic bulk DMI reads as
\begin{equation}
  \mathcal{G}_{\varepsilon} \left( \m \right) \assign \frac{1}{2 \varepsilon} 
  \sum_{i = 1}^3 \int_{\Om_{\varepsilon}} |\mathfrak{d}_i \m (x) |^2
  \hspace{0.17em} \mathrm{d} x,
\end{equation}
and is defined for every $\m \in H^1 (\Om_{\varepsilon}, \jM)$. Here, for $i =
1, 2, 3$, the quantities $\mathfrak{d}_i \m \assign \partial_i \m + \kappa
(\tmmathbf{e}_i \times \m)$ are usually referred to as {\tmem{helical
derivatives}}. The curved thin-film limit, as identified in
Theorem~\ref{thm:main}, is given by
\begin{align}
  \mathcal{E}_{\N} \left( \m \right)  \assign & \int_{\N} \left|  \grad_{\xi}
  \m (\xi) + \Aop_{\xi} (\m (\xi)) \right|^2 + \left( \Aop
  (\m (\xi)) \tmmathbf{n}_{\N} (\xi) \cdot \tmmathbf{n}_{\jM}
  (\m (\xi)) \right)^2 \mathd \xi \nonumber\\
  =&  \sum_{i = 1}^2 \int_{\N} \left| \partial_{\tmt_i (\xi)} \m (\xi) +
  \kappa \left( \tmt_i (\xi) \times \m (\xi) \right) \bp \right|^2 \mathd \xi
  \nonumber\\
  &   \qquad \qquad \qquad \qquad + \kappa^2 \int_{\N} \left( \left(
  \tmmathbf{n}_{\jM} ( \m (\xi) ) \times \m (\xi) \right) \cdot
  \tmmathbf{n}_{\N} (\xi) \right)^2 \mathd \xi .  \label{eq:1stcase}
\end{align}
When $\N \subseteq \RR^2 \times \{ 0 \}$ is a planar surface
({\tmabbr{cf.}}~Remark~\ref{rmk:moregensurf}), one has $\tmmathbf{n}_{\N}
(\xi) =\tmmathbf{e}_3$ and taking the tangential derivatives in the direction
of the standard basis of $\RR^2$, one obtains from \eqref{eq:1stcase} that
\begin{equation}
  \mathcal{E}_{\N} \left( \m \right) = \sum_{i = 1}^2 \int_{\N} \left|
  \mathfrak{d}_i \m (\xi) \bp \right|^2 \mathd \xi + \kappa^2 \int_{\N} \left(
  \left( \left. \m (\xi) \times \tmmathbf{n}_{\jM} \left( \m (\xi) \right.
  \right) \right) \cdot \tmmathbf{e}_3 \right)^2 \mathd \xi .
  \label{eq:1stcase2}
\end{equation}
Compared to Theorem 1 from {\cite{Davoli2022}}, we observe here that if $\m$
is $\Stwo^2$-valued, i.e., if $\jM = \Stwo^2$, then $\left. \m (\xi) \times
\tmmathbf{n}_{\Stwo^2} \left( \m (\xi) \right. \right) = \m (\xi) \times \m
(\xi) = 0$, and the limiting energy, as we already know, reduces to
\begin{equation}
  \mathcal{E}_{\N} \left( \m \right) = \sum_{i = 1}^2 \int_{\N} \left|
  \mathfrak{d}_i \m (\xi) \bp \right|^2 \mathd \xi .
\end{equation}
In other words, the second term in \eqref{eq:1stcase2} vanishes for
$\Stwo^2$-valued vector fields defined on a planar surface. Actually, the
second term in \eqref{eq:1stcase2} also vanishes when $\m$ is defined on a
closed surface $\N$ provided that $\m$ is still $\Stwo^2$-valued. However, and
this is an important point, when for physical reasons $\m$ cannot be
considered $\Stwo^2$-valued (e.g., because of anisotropy in the saturation
magnetization), then a correction must be taken into account, and this is
expressed by the second term in \eqref{eq:1stcase2}.

\subsection{Isotropic interfacial DMI}Let $\jM, \N$ be
two-dimensional surfaces in $\RR^3$. We are interested in the case where the
interfacial DMI occurs in the direction of the normal to $\N$. For that, we
observe that in terms of the moving frame
\begin{equation}
\left( \tmt_1 (\xi), \tmt_2 (\xi),
\tmmathbf{\tau}_3 (\xi) : =\tmmathbf{n}_{\N} (\xi) \right),
\end{equation}
the expression of
$\mathcal{H}_{\tmop{DM}}$ in \eqref{eq:H1asymDMI} can be written under the
form
\begin{equation}
  \mathcal{H}_{\tmop{DM}} \left( \m \right) = \sum_{i = 1}^3 \int_{\Om}
  \partial_{\tau_i (\xi)} \m (x) \cdot \left( \tmmathbf{d}_i (\xi) \times \m
  (x) \right) \mathd x \label{eq:H1asymDMItan}
\end{equation}
with $\xi \assign \pi_{\N} (x)$ and the Dzyaloshinskii vectors
$(\tmmathbf{d}_i (\xi))_{i = 1}^3$ depending on $\xi$. Hence, up to an
additive term, uninfluential for our purposes, the expression in
\eqref{eq:mainfunctionalinterest0} can be rewritten under the form
\begin{align}
  \mathcal{H}_{\Om} \left( \m \right) & \eqs \frac{1}{2}  \sum_{i = 1}^3
  \int_{\Om} | \partial_{\tau_i (\xi)} \m (x) +\tmmathbf{d}_i (\xi) \times \m
  (x) |^2 \mathd x \\
  & =  \frac{1}{2}  \sum_{i = 1}^3 \int_{\Om} | \mathD \m (x) T (\xi) + \Mop
  \left( \xi, \m (x) \right) |^2 \mathd x \, , \label{eq:tempcurvDMIint1}
\end{align}
with $T (\xi)$ the orthogonal operator mapping the standard basis of $\RR^3$
into the moving frame $\left( \tmt_1 (\xi), \tmt_2 (\xi), \tmmathbf{\tau}_3
(\xi) : =\tmmathbf{n}_{\N} (\xi) \right)$ and $\Mop \left( \xi, \m \right) =
\left( \tmmathbf{d}_1 (\xi) \times \m, \tmmathbf{d}_2 (\xi) \times \m,
\tmmathbf{d}_3 (\xi) \times \m \right)$. In the presence of curved interfacial
DMI, the Dzyaloshinskii vectors have the form 
\begin{equation}
\tmmathbf{d}_i (\xi) \assign
\kappa \left( \tmmathbf{n}_{\N} (\xi) \times \tmmathbf{\tau}_i (\xi) \right)\,
\end{equation}
so that (note that $\tmmathbf{d}_3 (\xi) =\tmmathbf{0}$) for $i = 1, 2$, there
holds
\begin{equation}
  \tmmathbf{d}_i (\xi) \times \m (x) \eqs \kappa \left[ \left( \m (x) \cdot
  \tmmathbf{n}_{\N} (\xi) \right) \tmmathbf{\tau}_i (\xi) - (\tmmathbf{m} (x)
  \cdot \tmmathbf{\tau}_i (\xi)) \tmmathbf{n}_{\N} (\xi) \right] .
\end{equation}
Equivalently, we have $\Mop \left( \xi, \m \right) = \Aop \left( \xi, \m
\right) \circ T (\xi)$ with $\Aop \left( \xi, \m \right)$ being the
matrix-valued function
\begin{equation}
  \Aop \left( \xi, \m \right) \assign \kappa \left[ \left( \bullet \, \otimes
  \tmmathbf{n}_{\N} (\xi) \right) - \left( \tmmathbf{n}_{\N} (\xi) \otimes
  \bullet \right) \right] \m . \label{eq:tempcurvDMIint2}
\end{equation}
Overall, from \eqref{eq:tempcurvDMIint1} and \eqref{eq:tempcurvDMIint2} and
the invariance of the norm under the orthogonal group, we get that
\begin{equation}
  \mathcal{H}_{\Om} \left( \m \right) = \frac{1}{2}  \sum_{i = 1}^3 \int_{\Om}
  | \mathD \m (x) + \Aop \left( \xi, \m \right) |^2 \mathd x.
\end{equation}
Coming back to the notation used for our asymptotic regime, the energy
functional of interest has the form ($\xi = \pi_{\N} (x)$)
\begin{equation}
  \tilde{\mathcal{G}}_{\varepsilon} \left( \m \right)  \assign  \frac{1}{2
  \varepsilon}  \sum_{i = 1}^3 \int_{\Om_{\varepsilon}} | \mathD \m (x) + \Aop
  \left( \xi, \m \right) |^2 \mathd x 
\end{equation}
and, again, it is defined for every $\m \in H^1 (\Om_{\varepsilon}, \jM)$.
Note that now $\Aop \left( \xi, \m \right)$ depends also on $\xi$, but the
analysis we performed to establish Theorem~\ref{thm:main} extends to this
setting (see Remark~\ref{rmk:curvinterf} and Theorem~\ref{thm:maingen}). The
curved thin-film limit, as identified in Theorem~\ref{thm:maingen}, reads as
\begin{align}
  \tilde{\mathcal{E}}_{\N} \left( \m \right) & \eqs \sum_{i = 1}^2
  \int_{\N} \left| \partial_{\tmmathbf{\tau}_i (\xi)} \m (\xi) + \Aop \left(
  \xi, \m (\xi) \right) \bp \tmt_i (\xi) \right|^2 \mathd \xi \nonumber\\
  &   \qquad \qquad \qquad \qquad + \int_{\N} \left( \Aop \left( \xi, \m
  (\xi) \right) \tmmathbf{n}_{\N} (\xi) \cdot \tmmathbf{n}_{\jM} \left( \m
  (\xi) \right) \right)^2 \mathd \xi  \label{eq:3rdcase0}\\
  & \eqs \sum_{i = 1}^2 \int_{\N} \left| \partial_{\tmt_i (\xi)} \m (\xi) +
  \kappa \left[ \left( \tmt_i (\xi) \otimes \tmmathbf{n}_{\N} (\xi) \right) -
  \left( \tmmathbf{n}_{\N} (\xi) \otimes \tmt_i (\xi) \, \right) \right] \m
  (\xi) \bp \right|^2 \mathd \xi .  \label{eq:3rdcase}
\end{align}
In fact, the second term in \eqref{eq:3rdcase0} vanishes because of $\Aop
\left( \xi, \m (\xi) \right) \tmmathbf{n}_{\N} (\xi) = 0$.

When $\N \subseteq \RR^2 \times \{ 0 \}$ is a flat surface
({\tmabbr{cf.}}~Remark~\ref{rmk:moregensurf}), one has $\tmmathbf{n}_{\N}
(\xi) =\tmmathbf{e}_3$. By taking the tangential
derivatives in the direction of the standard basis of $\RR^2$, we infer
from \eqref{eq:3rdcase} that
\begin{align}
  \tilde{\mathcal{E}}_{\N} \left( \m \right) & =  \sum_{i = 1}^2 \int_{\N} \left|
  \partial_i \m (\xi) + \kappa \left[ \left( \tmmathbf{e}_i \, \otimes
  \tmmathbf{e}_3 \right) - \left( \tmmathbf{e}_3 \otimes \tmmathbf{e}_i \,
  \right) \right] \m \bp \right|^2 \mathd \xi  \nonumber\\
  & \eqs  \int_{\N} \left| \nabla_{\xi}  \m \bp \right|^2 \mathd \xi + 2
  \kappa \int_{\N} m_3 \tmop{div}_{\xi}  \m - \m \cdot \grad_{\xi} m_3 \, \mathd
  \xi + \kappa^2 {\int_{\N}}  1 +
  m_3^2 (\xi) \, \mathd \xi.  \label{eq:2ndcase2}
\end{align}
The above formula is the classical expression for the interfacial DMI, which is
used in the analysis of magnetic skyrmions in planar thin films
{\cite{Muratov_2017}}.
We deduced \eqref{eq:2ndcase2} from \eqref{eq:3rdcase} and, therefore, from Theorem~\ref{thm:maingen}. 
It should be noted, however, that in the planar setting, invoking Theorem~\ref{thm:main} directly rather than its 
generalized counterpart stated in Theorem~\ref{thm:maingen} is sufficient to determine the thin-film limit.
In fact, when the surface is flat, the perturbation $\Aop$ does not depend on $\xi$.

\subsection{Anisotropic DMI}For completeness, we report the general case of
anisotropic DMI where the Dzyaloshinskii vectors $(\tmmathbf{d}_i)_{i = 1}^3$
are in general position. With $\jM, \N$ being two-dimensional surfaces in
$\RR^3$, the energy functional we are interested in reads as
\begin{equation}
  \mathcal{G}_{\varepsilon} \left( \m \right) = \frac{1}{2 \varepsilon}
  \int_{\Om_{\varepsilon}} \left| \mathD \m (x) + \Aop \left( \m (x) \right)
  \right|^2 \mathd x \eqs \frac{1}{2 \varepsilon} \sum_{i = 1}^3
  \int_{\Om_{\varepsilon}} \left| \partial_i \m (x) +\tmmathbf{d}_i \times \m
  (x) \right|^2 \mathd x, 
\end{equation}
with $\Aop \left( \m \right) = \Mop \left( \m \right)  \Jop$. Here, as in
\eqref{eq:mainfunctionalinterest}, $\Jop$ represents the linear operator that
maps the standard basis $(\tmmathbf{e}_i)_{i = 1}^3$ into the vectors
$(\tmmathbf{d}_i)_{i = 1}^3$, while $\Mop \left( \m \right) \assign \bullet \,
\times \m$ represents the $\m$-dependent {\tmem{antisymmetric}} linear
operator acting as the cross-product with $\m$.

By Theorem~\ref{thm:main}, we infer that in the curved thin-film limit,
anisotropic DMI contributions result in the limiting energy functional
\begin{equation}
  \tilde{\mathcal{E}}_{\calN} \left( \m \right) = \sum_{i = 1}^2 \int_{\N}
  \left| \partial_{\tmmathbf{\tau}_i (\xi)} \m (\xi) + \tilde{\tmmathbf{d}}_i
  (\xi) \times \m (\xi) \right|^2 + \int_{\N} \left( \Aop \left( \m (\xi)
  \right) \tmmathbf{n}_{\N} (\xi) \cdot \tmmathbf{n}_{\jM} \left( \m (\xi)
  \right) \right)^2 \mathd \xi  \label{eq:anisotropicDMI}
\end{equation}
with $\tilde{\tmmathbf{d}}_i (\xi) \assign \Jop \tmt_i (\xi)$. Observe that
when $M = \Stwo^2$, the second term in \eqref{eq:anisotropicDMI} vanishes.
Indeed, we have
\begin{equation}
  \Aop (\m (\xi)) \tmmathbf{n}_{\N} (\xi) \cdot
  \tmmathbf{n}_{\jM} (\m (\xi)) = \left[ \Mop \left( \m \right)
  \Jop \right] \tmmathbf{n}_{\N} (\xi) \cdot \m (\xi) = 0,
\end{equation}
because the image of $\Mop \left( \m \right) \Jop$ is orthogonal to $\m$.

\subsection{The nonuniform temperature setting}We discuss here the case when
the temperature of the ferromagnet is not uniform and, as a result, the
(temperature-dependent) saturation magnetization cannot be assumed constant.
To that purpose, we recall that in the variational theory of micromagnetism
({\tmabbr{cf.}} {\cite{LandauA1935,BrownB1963}}), systemized by Brown in the
60s, exchange interactions are described through the order parameter $\M$,
called the {\tmem{magnetization vector}}. The magnetization vector $\M$ of a
rigid ferromagnetic body, filling a domain $\Om \subseteq \RR^3$, \ is a
vector field satisfying the so-called {\tmem{saturation constraint}}, that is,
$| \M | = M_s$ in $\Om$ for a positive constant $M_s \geqslant 0$. The
{\tmem{saturation magnetization}} $M_s \assign M_s (T)$ is temperature
dependent and vanishes when $T$ exceeds the so-called {\tmem{Curie
temperature}} $T_c$, whose specific value depends on the crystal type. Mean
field theory predicts that the value of $M_s$ can be considered almost
constant in {\Om} when $T \ll T_c$ (see {\cite{Hubert1998,Kuz_min_2005}}). As a
result, it is customary to express the magnetization in the form $\M \assign
M_s \m$, where $\m : \Om \rightarrow \Stwo^2$ is now a vector field taking
values in the unit sphere $\Stwo^2$ of $\RR^3$; we did the same when
introducing the micromagnetic energy functionals in
\eqref{eq:vsrintercsmicromag} and \eqref{eq:mainfunctionalinterest0}.

Returning to our asymptotic regime, let $\jM, \N$ be two-dimensional surfaces
in $\RR^3$. In general, when the prescribed temperature profile is not uniform
in $\Om_{\varepsilon}$, $M_s$ is a function defined on $\Om_{\varepsilon}$,
and the family of energy functionals we are interested in has the form
({\tmabbr{cf.}}~with \eqref{eq:mainenfuncgen})
\begin{equation}
  \tilde{\mathcal{G}}_{\varepsilon} \left( \m \right) = \frac{1}{2
  \varepsilon} \sum_{i = 1}^3 \int_{\Om_{\varepsilon}} \left| \partial_i
  \left( M_s (x) \m (x) \right) + M_s (x) \tmmathbf{d}_i \times \m (x)
  \right|^2 \mathd x.
\end{equation}
In what follows, we assume that $M_s$ is $0$-homogeneous along the normal
direction, i.e., that $M_s (x) = M_s (\xi)$, $\xi \assign \pi_{\N} (x)$, and
that $M_s$ is regular enough. We can rearrange $\mathcal{G}$ in the following
way, which is compatible with Theorem~\ref{thm:maingen} (see also
\eqref{eq:mainenfuncgen}):
\begin{align}
  \tilde{\mathcal{G}}_{\varepsilon} \left( \m \right) & = \frac{1}{2
  \varepsilon} \int_{\Om_{\varepsilon}} \left| M_s (x) \partial_i \m (x) +
  \partial_i M_s (x)  \m (x) +\tmmathbf{d}_i \times \m (x) \right|^2
  \nonumber\\
  & =  \frac{1}{2 \varepsilon} \int_{\Om_{\varepsilon}} \left| A (x) \mathD
  \m (x) + \Aop \left( x, \m (x) \right) \right|^2 \mathd x 
\end{align}
with $A (x) = M_s (x) I$ and $\Aop \left( x, \m \right) =\tmmathbf{m} \otimes
\nabla M_s (x) + \Mop \left( \m \right)  \Jop$. Here, as in
\eqref{eq:mainfunctionalinterest}, $\Jop$ represents the linear operator that
maps the standard basis $(\tmmathbf{e}_i)_{i = 1}^3$ into the vectors
$(\tmmathbf{d}_i)_{i = 1}^3$, whereas $\Mop \left( \m \right) \assign \bullet
\, \times \m$ represents the $\m$-dependent {\tmem{antisymmetric}} linear
operator acting as the cross-product with $\m$. More explicitly, in matrix
form, we have
\begin{equation}
  \Aop \left( x, \m \right) = \left( \partial_1 M_s (x) \m, \partial_2 M_s (x)
  \m, \partial_3 M_s (x) \m \right) + \left( \tmmathbf{d}_1 \times \m,
  \tmmathbf{d}_2 \times \m, \tmmathbf{d}_3 \times \m \right) .
\end{equation}
Invoking Theorem~\ref{thm:maingen}, we infer that the curved thin-film limit
reads as
\begin{align}
    \tilde{\mathcal{E}}_{\calN} \left( \m \right)& =  \sum_{i = 1}^2 \int_{\N}
    \left| M_s (\xi) \partial_{\tmmathbf{\tau}_i (\xi)} \m (\xi) + \Aop \left(
    \xi, \m (\xi) \right) \bp \tmt_i (\xi) \right|^2 \mathd \xi \notag\\
    & \qquad \qquad \qquad + \int_{\N} \left( \Aop \left( \xi, \m (\xi)
    \right) \tmmathbf{n}_{\N} (\xi) \cdot \tmmathbf{n}_{\jM} \left( \m (\xi)
    \right) \right)^2 \mathd \xi \quad
\end{align}
Observe that when $M = \Stwo^2$, we have
\begin{equation}
  \Aop \left( \xi, \m (\xi) \right) \tmmathbf{n}_{\N} (\xi) \cdot
  \tmmathbf{n}_{\jM} (\m (\xi)) = \left[ \tmmathbf{m} \otimes
  \nabla_{\xi} M_s (\xi) + \Mop \left( \m \right) \Jop \right]
  \tmmathbf{n}_{\N} (\xi) \cdot \m (\xi) = 0,
\end{equation}
because the image of $\Mop \left( \m \right) \Jop$ is orthogonal to $\m$ and
$\nabla_{\xi} M_s (\xi)$ is tangent at $\xi \in \N$. Hence, in the classical
micromagnetic setting ($M = \Stwo^2$), we get that the limiting model reads as
\begin{equation}
  \tilde{\mathcal{E}}_{\calN} \left( \m \right) = \sum_{i = 1}^2 \int_{\N}
  \left| \partial_{\tmmathbf{\tau}_i (\xi)} \left( M_s (\xi) \m (\xi) \right)
  + \tilde{\tmmathbf{d}}_i (\xi) \times \m (\xi) \right|^2 \mathd \xi,
\end{equation}
with $\tilde{\tmmathbf{d}}_i (\xi) \assign \Jop \tmt_i (\xi)$.

\section*{Acknowledgments}

{\tmabbr{G.~Di~F.}} acknowledges support from the Austrian Science Fund (FWF)
through the project {\tmem{Analysis and Modeling of Magnetic Skyrmions}}
(grant P-34609). G. Di F. also thanks TU Wien and MedUni Wien for their
support and hospitality. The work of {\tmabbr{V.~S.}} was supported by the
Leverhulme grant RPG-2018-438.

\bibliography{literature.bib}
\bibliographystyle{siam}

\end{document}